\documentclass[3p, numbers]{elsarticle}
\usepackage{amsthm}
\usepackage{mathtools}
\usepackage{bbm}
\usepackage{tikz-cd}
\usepackage{hyperref}
\usepackage{csquotes}
\usepackage{amsmath}
\usepackage{amssymb}
\usepackage{amsfonts}
\usepackage{proof}
\usepackage{bussproofs}
\usepackage{mdframed}
\usepackage{float}

\newcommand{\m}{\mathbf}

\newcommand{\De}{\mathrm{\Delta}}
\newcommand{\Ga}{\mathrm{\Gamma}}

\newcommand*{\mh}{{\mathcal H}}
\newcommand*{\h}{\mid}

\newcommand{\topr}{\textup{\sc ($\top$)}}

\newcommand{\cutr}{\textup{\sc (cut)}}

\newcommand{\axr}{\textup{\sc (ax)}}
\newcommand{\cn}{\textup{\sc (c)}}

\newcommand{\weakr}{\textup{\sc (w)}}

\newcommand{\implrr}{\textup{\sc ($\to$r)}}
\newcommand{\implrrone}{\textup{\sc ($\to$r$_1$)}}
\newcommand{\implrrtwo}{\textup{\sc ($\to$r$_2$)}}
\newcommand{\conjlrone}{\textup{\sc ($\land$l$_1$)}}
\newcommand{\conjlrtwo}{\textup{\sc ($\land$l$_2$)}}
\newcommand{\conjlr}{\textup{\sc ($\land$l)}}
\newcommand{\disjrrone}{\textup{\sc ($\lor$r$_1$)}}
\newcommand{\disjrrtwo}{\textup{\sc ($\lor$r$_2$)}}
\newcommand{\disjlr}{\textup{\sc ($\lor$l)}}
\newcommand{\impllr}{\textup{\sc ($\to$l)}}
\newcommand{\conjrr}{\textup{\sc ($\land$r)}}
\newcommand{\copcone}{\textup{\sc (copc$_1$)}}
\newcommand{\copctwo}{\textup{\sc (copc$_2$)}}
\newcommand{\none}{\textup{\sc (n$_1$)}}
\newcommand{\ntwo}{\textup{\sc (n$_2$)}}
\newcommand{\nthree}{\textup{\sc (n$_3$)}}
\newcommand{\nfour}{\textup{\sc (n$_4$)}}
\newcommand{\copc}{\textup{\sc (copc)}}
\newcommand{\anr}{\textup{\sc (an)}}
\newcommand{\nr}{\textup{\sc (n)}}
\newcommand{\nefr}{\textup{\sc (nef)}}

\newcommand{\seq}{\Rightarrow}

\newcommand{\mpc}{{\sf MPC}}

\def\kn{\kern.1em}

\theoremstyle{plain}
\newtheorem{theorem}{Theorem}[section]
\newtheorem{corollary}{Corollary}[section]
\newtheorem{lemma}{Lemma}[section]

\newtheorem{prop}{Proposition}[section]

\theoremstyle{definition}

\newtheorem{example}{Example}[section]

\numberwithin{equation}{section}

\begin{document}
\begin{frontmatter}
\journal{ArXiv}
\title{Proof Theory for Positive Logic with Weak Negation}
\author[mb]{Marta~B\'ilkov\'a}
\ead{bilkova@cs.cas.cz}
\author[ac]{Almudena~Colacito}
\ead{almudena.colacito@math.unibe.ch}
\address[mb]{The Czech Academy of Sciences, Institute of Computer Science, Prague, Czech Republic}
\address[ac]{Mathematisches Institut, Universit\"at Bern, Bern, Switzerland}

\vspace*{-1.5cm}
\begin{abstract}
Proof-theoretic methods are developed for subsystems of Johansson's logic obtained by extending the positive fragment of intuitionistic logic with weak negations. These methods are exploited to establish properties of the logical systems. In particular, cut-free complete sequent calculi are introduced and used to provide a proof of the fact that the systems satisfy the Craig interpolation property. Alternative versions of the calculi are later obtained by means of an appropriate loop-checking history mechanism. Termination of the new calculi is proved, and used to conclude that the considered logical systems are PSPACE-complete.
\end{abstract}

\begin{keyword}
Minimal propositional logic \sep Weak negation \sep Intuitionistic propositional logic \sep Sequent calculus \sep Terminating sequent calculus \sep Decidability \sep Complexity

\end{keyword}
\end{frontmatter}


\section{Introduction}\label{s:introduction}

Minimal propositional calculus (\emph{Minimalkalk\"ul}, denoted here as \mpc) is the system obtained from the positive fragment of intuitionistic propositional calculus (equivalently, positive logic~\cite{rasiowa1974algebraic}) by adding a unary negation operator satisfying the so-called principle of contradiction (sometimes referred to as \emph{reductio ad absurdum}, e.g., in~\cite{odintsov2008constructive}). This system was introduced in this form by Johansson in 1937~\cite{johansson1937minimalkalkul}, and goes back to Kolmogorov's first formalization of intuitionistic logic~\cite{kolmogorov1925principle}, obtained by discarding \emph{ex falso quodlibet} (\emph{ex falso}, from now on) from the nowadays standard axioms for intuitionistic logic.~A letter from Johansson to Heyting (1935~-~1936) reads~\cite{vandermolen}:
\begin{displayquote}

[\emph{ex falso}] says that once $\neg a$ has been proved, $b$ follows from $a$, even if this had not been the case before.

\end{displayquote} 
This implies that the constructive interpretation of negation (i.e., implication) characteristic of intuitionism may give rise to doubts concerning the legitimacy of \emph{ex falso} as an axiom of intuitionistic logic. More generally, by rejecting \emph{ex falso}, one earns the right to study the notion of contradiction on its own and thereby, the related notion of negation.

The axiomatization proposed by Johansson preserves the whole positive fragment and \emph{most of} the negative fragment of Heyting's intuitionistic logic. As a matter of fact, many important properties of negation provable in Heyting's system remain provable (in some cases, in a slighlty weakened form) in minimal logic. The absence of \emph{ex falso} made Johansson's system the focus of interest in the field of paraconsistency, conceived as the study of those logics which admit inconsistent non-trivial theories. From a standard paraconsistent view, minimal logic still has unfortunate features~\cite{odintsov2008constructive}. In fact, the provability of what we are going to refer to as `negative \emph{ex falso}' $(a \land \neg a) \to \neg b$ makes negation meaningless in inconsistent theories---since every negated formula is provable---hence preserving some of the trivial aspects distinctive of \emph{ex falso}. Interestingly enough, in the setting of positive logic, negative \emph{ex falso} follows already from the assumption that negation is a functional antitone operator, i.e., that it satisfies the contraposition axiom $(a \to b) \to (\neg b \to \neg a)$~\cite{Colacito:Thesis:2016, colacito2016subminimal}. The latter is a theorem of Johansson's logic and one of its main `sources of explosiveness'\footnote{Logical systems in which all the inconsistent theories are trivial are sometimes said to be \emph{explosive}~\cite{odintsov2008constructive}. The fact that negative \emph{ex falso} is seen as a weak form of explosion justifies this statement.}. The two formulas mentioned in this paragraph will play a fundamental r\^ole throughout the paper. These formulas arose as central principles in the semantic investigations of~\cite{Colacito:Thesis:2016, colacito2016subminimal}. There, the class of those logics is studied that extend a basic logic of a unary operator axiomatized by $(a \leftrightarrow b) \to (\neg a \leftrightarrow \neg b)$ over positive logic. The present paper can be seen as a continuation of the study of these logics in a proof-theoretic direction. It is worth mentioning that the system with contraposition as the only negation axiom (here, {\sf CoPC}) has been considered as a theory of subminimal negation before: Hazen~\cite{Haz95}, following a personal communication with Humberstone, arrives at this logics by weakening Johansson's original motivations, avoiding in particular the principle $(a \to\neg a)\to\neg a$. The same logic has also been discussed by Humberstone himself in~\cite[Section 8.33]{Hum11}.\footnote{We thank Allen~P.~Hazen, and consequently Dick de Jongh, for pointing out and discussing this reference.}

Working in a paraconsistent setting and with weaker notions of negation sounds appealing when dealing with phenomena such as descriptions of counterfactual situations or fictional objects, information in a computer data base, negations in natural language, etc. For instance, fragments of intuitionistic logic have been considered to formalize information flow or access control, mostly because of their good balance between expressivity and tractability. A relevant example in this direction can be found looking at the \emph{logic of infons}~\cite{GN11}, used in \emph{Distributed Knowledge Authorization Language} (DKAL): `infons' are statements viewed as containers of information, and infon logic turns out to be a conservative extension of the positive and disjunction-free fragment of intuitionistic logic with certain simple multiagent modalities.

A formal study of the properties of (constructive) negation taken individually has the potential of disclosing unexplored paths in the study of negation in these more applied contexts. In this respect, we would like to point out the possibility of exploiting these studies in the setting of \emph{inquisitive logic}~\cite{groenendijk2009inquisitive, ciardelli2012inquisitive, ciardelli2013inquisitive}. The propositional fragment of inquisitive logic can be regarded as the disjunctive-negative fragment of intuitionistic logic, and negation plays a fundamental r\^ole in inquisitive semantics. Nonetheless, having a univocal notion of negation does not make it easy to model the many different shades of negation within natural language. An attempt to extending the inquisitive logic framework by adding a `weak' negation has been already pursued in~\cite{punvcochavr2015weak}, where to the persistent negation a weaker non-persistent negation is added, with the aim of taking care of problems as `might'-sentences or sentences expressing denial of conditionals. 

This short introduction is meant to outline where the study of \emph{weak} negations and, in particular, our work find their roots and motivations. The search for a `most appropriate' or `significant' notion of constructive negation given a fixed positive setting suggests a broad but well-defined realm of possibility, whose uniform study is typical of algebraic semantics. An algebraic presentation of the logical systems of our interest is the focus of Section~\ref{s:prelim}. This section is a condensed presentation of the results of~\cite{Colacito:Thesis:2016,BCdJ18}, and earlier papers on which these articles are based. After this, the remaining and main part of the paper is concerned with the development of a proof theory for the relevant systems. While this proof-theoretic account lacks uniformity, the sequent calculi that we present happen to be the most natural ones, obtained via a straightforward translation of equations (or equivalently, Hilbert-style axioms) into sequent rules. First, cut-free Gentzen systems are presented in Section~\ref{s:2}, and used in Section~\ref{s:3} to argue for fundamental results about the logical systems, such as Craig's interpolation, disjunction property, decidability. Later, the non-termination issues affecting the Gentzen systems are solved by means of a loop-checking history mechanism, presented and justified in Section~\ref{s:4}. The latter concludes the paper, as the termination property of the modified calculi allows us to improve the decidability result from Section~\ref{s:3} to an argument that the considered logical systems are decidable in PSPACE.

\section{Preliminaries}\label{s:prelim}

In this preliminary section we introduce the setting and present the main technical tools that will be used throughout the paper.

An algebraic semantics for minimal logic is given by the variety of contrapositionally complemented lattices. A \emph{contrapositionally complemented lattice} is an algebraic structure $\langle A, \land, \lor, \to, 1, \neg \rangle$, where $\langle A, \land, \lor, \to, 1 \rangle$ is a relatively pseudo-complemented lattice\footnote{In the sense of~\cite{rasiowa1974algebraic}; alternatively called Brouwerian algebras (e.g.,~\cite{moraschinisurjectivity}) or lattices (e.g.,~\cite{birkhoff1967lattice}), generalized Heyting algebras (e.g.,~\cite{ertola2007compatible, ertola2009subminimal}), implicative lattices (e.g.,~\cite{odintsov2008constructive}).} and the unary fundamental operation $\neg$ satisfies the identity $(p \to \neg q) \to (q \to \neg p) \approx 1$. A first algebraic account of Johansson's logic can be found in Rasiowa's main work on non-classical logics~\cite{rasiowa1974algebraic}.  The variety presented by Rasiowa is term-equivalent to the variety of relatively pseudo-complemented lattices with a negation operator defined by the algebraic formulation of the principle of non-contradiction $(p \to q) \land (p \to \neg q) \to \neg p \approx 1$, originally employed in Johansson's axiomatization. Observe that Heyting algebras can be seen as contrapositionally complemented lattices where $\neg 1$ is a distinguished bottom element $0$. 

As stated in the introduction, we focus on subsystems of minimal logic obtained by weakening Johansson's notion of negation. Algebraically, we start from a class of algebraic structures $\langle A, \land, \lor, \to, 1, \neg \rangle$ that generalize contrapositionally complemented lattices~\cite{rasiowa1974algebraic}. Again $\langle A, \land, \lor, \to, 1 \rangle$ is a relatively pseudo-complemented lattice and the unary operation $\neg$ is restricted only by the identity: 
\begin{equation}\label{eq:n}
(p \leftrightarrow q) \to ( \neg p \leftrightarrow \neg q) \approx 1. 
\end{equation}
We call these structures \emph{N-algebras}, and we refer to the equivalent logical system as {\sf N}. Note that this class of algebras play a fundamental r\^ole in the attempts of defining a connective over positive logic. In fact, the considered equation states that the function $\neg$ is a \emph{compatible function} (or \emph{compatible connective}), in the sense that every
congruence of $\langle A, \land, \lor, \to, 1 \rangle$ is a congruence of $\langle A, \land, \lor, \to, 1, \neg \rangle$. This is somehow considered a minimal requirement when introducing a new connective over a fixed setting. A study of compatible connectives on Heyting algebras can be found in~\cite{caicedo2001algebraic}, in the attempt of defining and studying new connectives over intuitionistic logic. In the setting of positive logic, a similar algebraic approach is carried out in~\cite{ertola2007compatible}. 

Given the variety of N-algebras, we are particularly interested in two of its subvarieties that were implicitly considered in Section~\ref{s:introduction}, and are defined, respectively, by the following identities: 

\begin{equation}\label{eq:nef}
(p \land \neg p) \to \neg q \approx 1
\end{equation}

\begin{equation}\label{eq:copc}
(p \to q) \to (\neg q \to \neg p) \approx 1.
\end{equation}

\noindent Observe that~\ref{eq:nef} and~\ref{eq:copc} are algebraic formulations of, respectively, negative \emph{ex falso} and the contraposition axiom, and for this reason we shall refer to the corresponding logics as \emph{negative ex falso} logic ({\sf NeF}) and \emph{contraposition} logic ({\sf CoPC}). Here, we also consider contrapositionally complemented lattices in terms of their term-equivalent variety of N-algebras defined by $(p \to q) \to (\neg q \to \neg p) \approx 1$ plus $(p \to \neg p) \to \neg p \approx 1$. It was proved in~\cite{Colacito:Thesis:2016,colacito2016subminimal,BCdJ18} that the following relations hold between the considered logical systems:
\[
{\sf N} \subset {\sf NeF} \subset {\sf CoPC} \subset {\sf MPC},
\]
where the strict inclusion ${\sf L}_1 \subset {\sf L}_2$ means that every theorem of ${\sf L}_1$ is a theorem of ${\sf L}_2$, and at least one theorem of ${\sf L}_2$ is not provable in ${\sf L}_1$.

In the remainder of this paper we focus on a proof-theoretic study of (the logics corresponding to) these varieties. We will talk interchangeably about logical systems and their corresponding algebraic presentation, assuming an algebraic completeness result that goes beyond the scope of this paper and is therefore not included here. In order to keep the structure of the paper as clean as possible, the reader interested in a broader and introductory account of the topic from an algebraic perspective should refer to~\cite{Colacito:Thesis:2016}. It is worth mentioning that in~\cite{Colacito:Thesis:2016, colacito2016subminimal}, the presentation and the study of these logical systems are developed also in a Kripke-style semantical framework. Finally,~\cite{BCdJ18} is concerned with a deeper algebraic account of these logics, and a first study of their modal companions.


\section{Sequent Calculi}\label{s:2}

Following the notation of~\cite{troelstra2000basic}, we refer to the Gentzen systems introduced here as \textbf{G3}-systems. The letters $p, q, \ldots$ range over propositional variables, the Greek letters $\alpha, \beta, \ldots, \varphi, \ldots$ range over formulas, and $\Ga, \De, \ldots$ range over finite multisets of formulas. We write $\Ga, \De$ for the multiset obtained as the union of $\Ga$ and $\De$. We consider the propositional language $\mathcal{L}({\sf Prop})$, where ${\sf Prop}$ is a countable set of propositional variables, generated by the following grammar: 
\[
p \mid \varphi \land \varphi \mid \varphi \lor \varphi \mid \varphi \to \varphi \mid \neg \varphi
\]  
\noindent 
where $p \in {\sf Prop}$. We consider $\top$ as a constant in the language, while we omit $\bot$. We call a formula \emph{positive} if it contains only connectives from $\{\land, \lor, \to, \top\}$, and we refer to the positive fragment of intuitionistic logic as \emph{positive logic}. We define the notion of \emph{weight} (or \emph{complexity}) $w(\varphi)$ of a formula $\varphi$ inductively as follows: 

\begin{itemize}

\item $w(p) = w(\top) = 1$, for every atom $p$; 

\item $w(\varphi_1 \circ \varphi_2) = w(\varphi_1) + w(\varphi_2) + 1$, where $\circ \in \{\land, \lor, \to\}$;

\item $w(\neg\varphi) = w(\varphi) + 1$.

\end{itemize}
\begin{figure}[h]
\begin{mdframed}[style=mdfexample1]
\[
\textbf{Logical Rules for Positive Connectives}
\]
\[
 \begin{array}{ccccccccc}
 \infer[\axr]{\Ga, p \Rightarrow p}{} & 
 \qquad &
 \infer[{\topr}]{\Ga \Rightarrow \top}{} &
 \end{array}
\]
\[
 \begin{array}{ccccccccc}
 \infer[\implrr]{\Ga \Rightarrow \alpha\to\beta}{\Ga, \alpha \Rightarrow \beta} 
 \qquad & 
 \infer[\impllr]{\Ga, \alpha\to\beta \Rightarrow \varphi}{\Ga, \alpha\to\beta \Rightarrow \alpha \qquad \Ga, \beta \Rightarrow \varphi} 
 \end{array}
\]
\[
 \begin{array}{ccccccccc}
 \infer[\conjrr]{\Ga \Rightarrow \alpha\land\beta}{\Ga \Rightarrow \alpha \qquad \Ga \Rightarrow \beta} 
 \end{array}
\]
\[
 \begin{array}{ccccccccc}
 \infer[\conjlr]{\Ga, \alpha\land\beta \Rightarrow \varphi}{\Ga, \alpha, \beta \Rightarrow \varphi} 
 \end{array}
\]
\[
 \begin{array}{ccccccccc}
 \infer[\disjrrone]{\Ga \Rightarrow \alpha\lor\beta}{\Ga \Rightarrow \alpha} 
 \qquad &
 \infer[\disjrrtwo]{\Ga \Rightarrow \alpha\lor\beta}{\Ga \Rightarrow \beta} 
 \end{array}
\]
\[
 \begin{array}{ccccccccc}
 \infer[\disjlr]{\Ga, \alpha\lor\beta \Rightarrow \varphi}{\Ga, \alpha \Rightarrow \varphi \qquad \Ga, \beta \Rightarrow \varphi} 
 \end{array}
\]
\[
\textbf{Logical Rules for Negation}
\]
\[
 \begin{array}{ccccccccc}
 \infer[\nr]{\Ga, \neg\alpha \Rightarrow \neg\beta}{\Ga, \neg\alpha, \beta \Rightarrow \alpha \qquad \Ga, \neg\alpha, \alpha \Rightarrow \beta} 
 \end{array}
\]
\[
 \begin{array}{ccccccccc}
 \infer[\nefr]{\Ga, \neg\alpha \Rightarrow \neg\beta}{\Ga, \neg\alpha \Rightarrow \alpha} 
 \end{array}
\]
\[
 \begin{array}{ccccccccc}
 \infer[\copc]{\Ga, \neg\alpha \Rightarrow \neg\beta}{\Ga, \neg\alpha, \beta \Rightarrow \alpha} 
 \qquad & 
 \infer[\anr]{\Ga \Rightarrow \neg\alpha}{\Ga, \alpha \Rightarrow \neg\alpha} 
 \end{array}
\]
\end{mdframed}
\caption{The systems $\m{G3}$} \label{fig:g}
\end{figure}

\begin{figure}[h]
\begin{mdframed}[style=mdfexample1]
\[
 \begin{array}{ccccccccc}
 \infer[\weakr]{\Ga, \alpha \Rightarrow \varphi}{\Ga \Rightarrow \varphi} 
 \qquad & 
 \infer[\cn]{\Ga, \alpha \Rightarrow \varphi}{\Ga, \alpha, \alpha \Rightarrow \varphi} 
 \qquad &
 \infer[\cutr]{\Ga, \De \Rightarrow \varphi}{\Ga \Rightarrow \alpha \quad \De, \alpha \Rightarrow \varphi} 
 \end{array}
\]
\end{mdframed}
\caption{Structural rules} \label{fig:sr}
\end{figure}

In Figures~\ref{fig:g} and~\ref{fig:sr}, four different sequent systems are presented. Keeping the rules for the positive connectives and the structural rules fixed, the rule $\nr$ defines a proof system for N-algebras; by adding the rule $\nefr$ and by substituting the rule $\nr$ with $\copc$, we obtain systems for the two proper subvarieties of N-algebras defined, respectively, by~\ref{eq:nef} and~\ref{eq:copc}. Finally, by adding the rule $\anr$ to the system obtained with $\copc$, we obtain a sequent proof system for contrapositionally complemented lattices. 

Sequent proof systems for \mpc\ were already introduced, e.g., in~\cite{troelstra2000basic}, using a different presentation of \mpc. In fact, \mpc\ is often formulated with the negation $\neg \varphi$ as a defined connective $\varphi \to f$, where $f$ is a distinguished propositional variable in the language. This formulation naturally suggests a sequent calculus for \mpc\, obtained by discarding the axiom ruling the constant $\bot$ from the system for intuitionistic propositional logic.

We call the \emph{height} of a derivation $\vdash \Ga \Rightarrow \varphi$ the number of inference steps of a maximal branch within a proof tree with root $\Ga \Rightarrow \varphi$. A rule is said to be \emph{admissible} if, whenever its premises are derivable in the system, so is the conclusion, and \emph{height-preserving admissible} if, whenever its premises are derivable via a proof of height (at most) $n \in \mathbb{N}$, then so is the conclusion. We say that an inference rule is \emph{invertible} if, for every instance of the rule, the premises are derivable if, and only if, the conclusion is, and \emph{height-preserving invertible} if the two derivations have the same height. It can be proved that all the rules in Figure~\ref{fig:g} except $\disjrrone$, $\disjrrtwo$, $\impllr$, $\nr$, $\nefr$ and $\copc$ are height-preserving invertible. 

The arguments for admissibility of weakening $\weakr$ and contraction $\cn$ are standard, and for the positive part of the calculus they can be adopted directly from a number of sources. Therefore we concentrate on the negative rules only, and refer the reader to the textbook~\cite{troelstra2000basic} for the remaining cases. The full proofs can be found in~\cite{Colacito:Thesis:2016}. 

\begin{prop}\label{prop:w}
The following weakening rule is height-preserving admissible:
\[
\begin{array}{ccccccccc}
 \infer[\weakr]{\Ga, \alpha \Rightarrow \varphi}{\Ga \Rightarrow \varphi} 
\end{array}
\]
\end{prop}

\begin{proof}
The proof goes by induction on the height of a derivation of the premise sequent $\Ga \Rightarrow \varphi$. 

Suppose that the derivation (of height $n$) ends with an application of $\nr$
\[
\begin{array}{ccccccccc}
\infer[\nr]{\Ga, \neg\beta \Rightarrow \neg\gamma}{\Ga, \neg\beta, \beta \Rightarrow \gamma \qquad \Ga, \neg\beta, \gamma \Rightarrow \beta} 
\end{array}
\] 
By the induction hypothesis twice, we have derivations
\[
\vdash \Ga, \neg\beta, \beta, \alpha \Rightarrow \gamma\text{ and }\vdash \Ga, \neg\beta, \gamma, \alpha \Rightarrow \beta
\]
of height, respectively, at most $m_1$ and at most $m_2$, where $m_1$ and $m_2$ are the height of derivations of the corresponding premises. By using $\nr$, we get a derivation of height at most $n$ of the sequent $\Ga, \neg\beta, \alpha \Rightarrow \neg\gamma$.

Suppose that the derivation ends with an application of $\nefr$
\[
\begin{array}{ccccccccc}
\infer[\nefr]{\Ga, \neg\beta \Rightarrow \neg\gamma}{\Ga, \neg\beta \Rightarrow \beta} 
\end{array}
\] 
By the induction hypothesis, we have a derivation
\[
\vdash \Ga, \neg\beta, \alpha \Rightarrow \beta
\]
of height at most $n-1$. By using $\nefr$, we get a derivation of height at most $n$ of the sequent $\Ga, \neg\beta, \alpha \Rightarrow \neg\gamma$.

Suppose that the derivation ends with an application of $\copc$
\[
\begin{array}{ccccccccc}
\infer[\copc]{\Ga, \neg\beta \Rightarrow \neg\gamma}{\Ga, \neg\beta, \gamma \Rightarrow \beta} 
\end{array}
\] 
By the induction hypothesis, we have a derivation
\[
\vdash \Ga, \neg\beta, \gamma, \alpha \Rightarrow \beta
\]
of height at most $n-1$. By using $\copc$, we get a derivation of height at most $n$ of the sequent $\Ga, \neg\beta, \alpha \Rightarrow \neg\gamma$.

Finally, suppose that the derivation ends with an application of $\anr$
\[
\begin{array}{ccccccccc}
\infer[\anr]{\Ga \Rightarrow \neg\beta}{\Ga, \beta \Rightarrow \neg\beta} 
\end{array}
\] 
By the induction hypothesis, we have a derivation
\[
\vdash \Ga, \beta, \alpha \Rightarrow \neg\beta
\]
of height at most $n-1$. By using $\anr$, we get a derivation of height at most $n$ of the sequent $\Ga, \alpha \Rightarrow \neg\beta$.
\end{proof}

\begin{prop}
The following contraction rule is height-preserving admissible:
\[
\begin{array}{ccccccccc}
 \infer[\cn]{\Ga, \alpha \Rightarrow \varphi}{\Ga, \alpha, \alpha \Rightarrow \varphi} 
\end{array}
\]
\end{prop}

\begin{proof}
The proof goes by induction (with respect to the lexicographic order) on the pair $(n,m)$, where $n=w(\alpha)$ and $m$ is the height of a derivation of $\Ga, \alpha, \alpha \Rightarrow \varphi$.\footnote{The induction on the ordered pair is needed for the positive rules, e.g., the $\conjlr$-case with $\alpha$ principal (see~\cite{troelstra2000basic}; cf.\ \cite{Colacito:Thesis:2016}).} 

Suppose that the derivation ends with an application of $\nr$ where $\alpha$ is not a principal formula:
\[
\begin{array}{ccccccccc}
\infer[\nr]{\Ga, \neg\beta, \alpha, \alpha \Rightarrow \neg\gamma}{\Ga, \neg\beta, \beta, \alpha, \alpha \Rightarrow \gamma \qquad \Ga, \neg\beta, \gamma, \alpha, \alpha \Rightarrow \beta} 
\end{array}
\] 
By the induction hypothesis twice, we have derivations
\[
\vdash \Ga, \neg\beta, \beta, \alpha \Rightarrow \gamma\text{ and }\vdash \Ga, \neg\beta, \gamma, \alpha \Rightarrow \beta
\]
of height, respectively, at most $m_1$ and at most $m_2$, where $m_1$ and $m_2$ are the height of derivations of the corresponding premises. By using $\nr$, we get a derivation of height at most $n$ of the sequent $\Ga, \neg\beta, \alpha \Rightarrow \neg\gamma$.

Suppose that the derivation ends with an application of $\nr$ where $\alpha$ is a principal formula $\neg\alpha_1$:
\[
\begin{array}{ccccccccc}
\infer[\nr]{\Ga, \neg\alpha_1, \neg\alpha_1 \Rightarrow \neg\gamma}{\Ga, \neg\alpha_1, \neg\alpha_1, \alpha_1 \Rightarrow \gamma \qquad \Ga, \neg\alpha_1, \neg\alpha_1, \gamma \Rightarrow \alpha_1} 
\end{array}
\] 
By the induction hypothesis twice, we have derivations
\[
\vdash \Ga, \neg\alpha_1, \alpha_1 \Rightarrow \gamma\text{ and }\vdash \Ga, \neg\alpha_1, \gamma \Rightarrow \alpha_1
\]
of height, respectively, at most $m_1$ and at most $m_2$, where $m_1$ and $m_2$ are the height of derivations of the corresponding premises. By using $\nr$, we get a derivation of height at most $n$ of the sequent $\Ga, \neg\alpha_1 \Rightarrow \neg\gamma$.

Suppose that the derivation ends with an application of $\nefr$ where $\alpha$ is not a principal formula:
\[
\begin{array}{ccccccccc}
\infer[\nefr]{\Ga, \neg\beta, \alpha, \alpha \Rightarrow \neg\gamma}{\Ga, \neg\beta, \alpha, \alpha \Rightarrow \beta} 
\end{array}
\] 
By the induction hypothesis, we have a derivation
\[
\vdash \Ga, \neg\beta, \alpha \Rightarrow \beta
\]
of height at most $n-1$. By using $\nefr$, we get a derivation of height at most $n$ of the sequent $\Ga, \neg\beta, \alpha \Rightarrow \neg\gamma$.

Suppose that the derivation ends with an application of $\nefr$ where $\alpha$ is a principal formula $\neg\alpha_1$:
\[
\begin{array}{ccccccccc}
\infer[\nefr]{\Ga, \neg\alpha_1, \neg\alpha_1 \Rightarrow \neg\gamma}{\Ga, \neg\alpha_1, \neg\alpha_1 \Rightarrow \alpha_1} 
\end{array}
\] 
By the induction hypothesis, we have a derivation
\[
\vdash \Ga, \neg\alpha_1 \Rightarrow \alpha_1
\]
of height at most $n-1$. By using $\nefr$, we get a derivation of height at most $n$ of the sequent $\Ga, \neg\alpha_1 \Rightarrow \neg\gamma$.

Suppose that the derivation ends with an application of $\copc$ where $\alpha$ is not a principal formula:
\[
\begin{array}{ccccccccc}
\infer[\copc]{\Ga, \neg\beta, \alpha, \alpha \Rightarrow \neg\gamma}{\Ga, \neg\beta, \gamma, \alpha, \alpha \Rightarrow \beta} 
\end{array}
\] 
By the induction hypothesis, we have a derivation
\[
\vdash \Ga, \neg\beta, \gamma, \alpha \Rightarrow \beta
\]
of height at most $n-1$. By using $\copc$, we get a derivation of height at most $n$ of the sequent $\Ga, \neg\beta, \alpha \Rightarrow \neg\gamma$.

Suppose that the derivation ends with an application of $\copc$ where $\alpha$ is a principal formula $\neg\alpha_1$:
\[
\begin{array}{ccccccccc}
\infer[\copc]{\Ga, \neg\alpha_1, \neg\alpha_1 \Rightarrow \neg\gamma}{\Ga, \neg\alpha_1, \neg\alpha_1, \gamma \Rightarrow \alpha_1} 
\end{array}
\] 
By the induction hypothesis, we have a derivation
\[
\vdash \Ga, \neg\alpha_1, \gamma \Rightarrow \alpha_1
\]
of height at most $n-1$. By using $\copc$, we get a derivation of height at most $n$ of the sequent $\Ga, \neg\alpha_1 \Rightarrow \neg\gamma$.

Finally, suppose that the derivation ends with an application of $\anr$
\[
\begin{array}{ccccccccc}
\infer[\anr]{\Ga, \alpha, \alpha \Rightarrow \neg\beta}{\Ga, \beta, \alpha, \alpha \Rightarrow \neg\beta} 
\end{array}
\] 
By the induction hypothesis, we have a derivation
\[
\vdash \Ga, \beta, \alpha \Rightarrow \neg\beta
\]
of height at most $n-1$. By using $\anr$, we get a derivation of height at most $n$ of the sequent $\Ga, \alpha \Rightarrow \neg\beta$.
\end{proof}

A reader used to dealing with the sequent calculus for intuitionistic logic won't find it hard to believe the following fundamental result~\cite{Colacito:Thesis:2016}.

\begin{theorem}[Soundness and Completeness]\label{t:completeness}
Given a finite multiset of formulas $\Ga, \varphi$, the sequent $\Ga \Rightarrow \varphi$ is derivable in the \textbf{G3}-system if, and only if, $\bigwedge \Ga \to \varphi \approx 1$ is a valid equation in the corresponding variety of N-algebras.
\end{theorem}


\section{Elimination of Cut}\label{s:3}

It is often said that finding the right (i.e., sound and complete) set of rules for a logical system requires a certain amount of ingenuity. In order to carry out proof search in the sequent system, and to use the calculus to understand the corresponding logic, we seek to eliminate the cut rule. The family of calculi presented in Section~\ref{s:2} is obtained by translating in the most natural way equations into sequent rules, and consists of `good' systems, in the sense that they are cut-free. We don't go into the details of the proof of cut admissibility, that are somehow standard and can be found in full in~\cite{Colacito:Thesis:2016}. We will however outline the two interesting cases in the inductive proof---concerning the negative rules---and later focus on some relevant consequences ensuing from cut elimination.

\begin{theorem}
The following cut rule is admissible:
\[
\begin{array}{ccccccccc}
\infer[\cutr]{\Ga, \De \Rightarrow \varphi}{\Ga \Rightarrow \alpha \qquad \De, \alpha \Rightarrow \varphi} 
\end{array}
\]
\end{theorem}

\begin{proof}
The proof goes by induction (with respect to the lexicographic order) on the pair $(n,m)$, where $n=w(\alpha)$ and $m$ is the combined heights of derivations of the premises, i.e., if $\vdash_{m_1} \Ga \Rightarrow \alpha$ and $\vdash_{m_2} \De,\alpha \Rightarrow \varphi$, then $m=m_1+m_2$. We focus on the negative rules, and consider two cases: the case in which (only) the left premise is an axiom, and the case in which neither of the premises is an axiom and the formula $\alpha$ is principal in both of them. For the remainig cases, we refer to~\cite{troelstra2000basic} and~\cite{Colacito:Thesis:2016}.

Consider the case in which the left premise is an axiom. 

Suppose that the last step in the derivation of the right premise is $\nr$:
\[
\begin{array}{ccccccccc}
\infer[\nr]{\De, \neg \beta, \alpha \Rightarrow \neg\gamma}{\De, \neg \beta, \beta, \alpha \Rightarrow \gamma \qquad \De, \neg \beta, \gamma, \alpha \Rightarrow \beta} 
\end{array}
\]
By the induction hypothesis twice, we have derivations
\[
\begin{array}{ccccccccc}
\infer[]{\Ga, \De, \neg \beta, \gamma \Rightarrow \beta}{\Ga \Rightarrow \alpha \qquad \De, \neg \beta, \gamma, \alpha \Rightarrow \beta} 
\end{array}
\]
and
\[
\begin{array}{ccccccccc}
\infer[]{\Ga, \De, \neg \beta, \beta \Rightarrow \gamma}{\Ga \Rightarrow \alpha \qquad \De, \neg \beta, \beta, \alpha \Rightarrow \gamma} 
\end{array}
\]
By using $\nr$, we get a derivation of the sequent $\Ga, \De, \neg\beta \Rightarrow \neg\gamma$.

Suppose that the last step in the derivation of the right premise is $\nefr$:
\[
\begin{array}{ccccccccc}
\infer[\nefr]{\De, \neg \beta, \alpha \Rightarrow \neg\gamma}{\De, \neg \beta, \alpha \Rightarrow \beta} 
\end{array}
\]
By the induction hypothesis, we have a derivation
\[
\begin{array}{ccccccccc}
\infer[]{\Ga, \De, \neg \beta \Rightarrow \beta}{\Ga \Rightarrow \alpha \qquad \De, \neg \beta, \alpha \Rightarrow \beta} 
\end{array}
\]
By using $\nefr$, we get a derivation of the sequent $\Ga, \De, \neg\beta \Rightarrow \neg\gamma$.

Suppose that the derivation ends with an application of $\copc$:
\[
\begin{array}{ccccccccc}
\infer[\copc]{\De, \neg \beta, \alpha \Rightarrow \neg\gamma}{\De, \neg \beta, \gamma, \alpha \Rightarrow \beta} 
\end{array}
\]
By the induction hypothesis, we have a derivation
\[
\begin{array}{ccccccccc}
\infer[]{\Ga, \De, \neg \beta, \gamma \Rightarrow \beta}{\Ga \Rightarrow \alpha \qquad \De, \neg \beta, \gamma, \alpha \Rightarrow \beta} 
\end{array}
\]
By using $\copc$, we get a derivation of the sequent $\Ga, \De, \neg\beta \Rightarrow \neg\gamma$.

Suppose that the last step in the derivation of the right premise is $\anr$:
\[
\begin{array}{ccccccccc}
\infer[\anr]{\De, \alpha \Rightarrow \neg \beta}{\De, \beta, \alpha \Rightarrow \neg\beta} 
\end{array}
\]
By the induction hypothesis, we have a derivation
\[
\begin{array}{ccccccccc}
\infer[]{\Ga, \De, \beta \Rightarrow \neg\beta}{\Ga \Rightarrow \alpha \qquad \De, \beta, \alpha \Rightarrow \neg\beta} 
\end{array}
\]
By using $\anr$, we get a derivation of the sequent $\Ga, \De \Rightarrow \neg\beta$.

Consider the case in which neither of the premises is an axiom, and the formula $\alpha$ is principal in both of them. The left sequent comes as the conclusion of a negation rule if and only if the right premise sequent does.

Suppose that the last step in the derivation of both premises is $\nr$:
\[
\begin{array}{ccccccccc}
\infer[\nr]{\Ga, \neg \beta \Rightarrow \neg\alpha_1}{\Ga, \neg \beta, \beta \Rightarrow \alpha_1 \qquad \Ga, \neg \beta, \alpha_1 \Rightarrow \beta} 
\end{array}
\]
\[
\begin{array}{ccccccccc}
\infer[\nr]{\De, \neg \alpha_1 \Rightarrow \neg\gamma}{\De, \neg \alpha_1, \alpha_1 \Rightarrow \gamma \qquad \De, \neg \alpha_1, \gamma \Rightarrow \alpha_1} 
\end{array}
\]
By the induction hypothesis twice, we have derivations
\[
\begin{array}{ccccccccc}
\infer[]{\Ga, \De, \neg\beta, \alpha_1 \Rightarrow \gamma}{\Ga, \neg \beta \Rightarrow \neg\alpha_1 \qquad \De, \neg \alpha_1, \alpha_1 \Rightarrow \gamma} 
\end{array}
\]
and
\[
\begin{array}{ccccccccc}
\infer[]{\Ga, \De, \neg\beta, \gamma \Rightarrow \alpha_1}{\Ga, \neg \beta \Rightarrow \neg\alpha_1 \qquad \De, \neg \alpha_1, \gamma \Rightarrow \alpha_1} 
\end{array}
\]
Again, by the induction hypothesis twice, and the contraction rule, we get
\[
\begin{array}{ccccccccc}
\infer[]{\Ga, \De, \neg\beta, \beta \Rightarrow \gamma}{\Ga, \neg \beta, \beta \Rightarrow \alpha_1 \qquad \Ga, \De, \neg\beta, \alpha_1 \Rightarrow \gamma} 
\end{array}
\]
and
\[
\begin{array}{ccccccccc}
\infer[]{\Ga, \De, \neg\beta, \gamma \Rightarrow \beta}{\Ga, \De, \neg\beta, \gamma \Rightarrow \alpha_1 \qquad \Ga, \neg \beta, \alpha_1 \Rightarrow \beta} 
\end{array}
\]
Finally, by using $\nr$, we get a derivation of the sequent $\Ga, \De, \neg\beta \Rightarrow \neg\gamma$.

Suppose that the last step in the derivation of the left premise is $\nr$, and the right premise comes from an application of $\nefr$:
\[
\begin{array}{ccccccccc}
\infer[\nr]{\Ga, \neg \beta \Rightarrow \neg\alpha_1}{\Ga, \neg \beta, \beta \Rightarrow \alpha_1 \qquad \Ga, \neg \beta, \alpha_1 \Rightarrow \beta} 
\end{array}
\]
\[
\begin{array}{ccccccccc}
\infer[\nefr]{\De, \neg \alpha_1 \Rightarrow \neg\gamma}{\De, \neg \alpha_1 \Rightarrow \alpha_1} 
\end{array}
\]
By the induction hypothesis, we have a derivation
\[
\begin{array}{ccccccccc}
\infer[]{\Ga, \De, \neg\beta \Rightarrow \alpha_1}{\Ga, \neg \beta \Rightarrow \neg\alpha_1 \qquad \De, \neg \alpha_1 \Rightarrow \alpha_1} 
\end{array}
\]
But then, by the induction hypothesis, and the contraction rule, we get
\[
\begin{array}{ccccccccc}
\infer[]{\Ga, \De, \neg\beta \Rightarrow \beta}{\Ga, \De, \neg\beta \Rightarrow \alpha_1 \qquad \Ga, \neg \beta, \alpha_1 \Rightarrow \beta} 
\end{array}
\]
Finally, by using $\nefr$, we get a derivation of the sequent $\Ga, \De, \neg\beta \Rightarrow \neg\gamma$.

Suppose that the last step in the derivation of both premises is $\copc$:
\[
\begin{array}{ccccccccc}
\infer[\copc]{\Ga, \neg \beta \Rightarrow \neg\alpha_1}{\Ga, \neg \beta, \alpha_1 \Rightarrow \beta} 
\end{array}
\]
\[
\begin{array}{ccccccccc}
\infer[\copc]{\De, \neg \alpha_1 \Rightarrow \neg\gamma}{\De, \neg \alpha_1, \gamma \Rightarrow \alpha_1} 
\end{array}
\]
By the induction hypothesis, we have a derivation
\[
\begin{array}{ccccccccc}
\infer[]{\Ga, \De, \neg\beta, \gamma \Rightarrow \alpha_1}{\Ga, \neg \beta \Rightarrow \neg\alpha_1 \qquad \De, \neg \alpha_1, \gamma \Rightarrow \alpha_1} 
\end{array}
\]
By the induction hypothesis again, and the contraction rule, we also get
\[
\begin{array}{ccccccccc}
\infer[]{\Ga, \De, \neg\beta, \gamma \Rightarrow \beta}{\Ga, \De, \neg\beta, \gamma \Rightarrow \alpha_1 \qquad \Ga, \neg \beta, \alpha_1 \Rightarrow \beta} 
\end{array}
\]
Hence, by using $\copc$, we get a derivation of the sequent $\Ga, \De, \neg\beta \Rightarrow \neg\gamma$. 

Suppose the derivation of the left premise ends with an application of $\anr$, and the right premise comes from an application of $\copc$:
\[
\begin{array}{ccccccccc}
\infer[\anr]{\Ga \Rightarrow \neg\alpha_1}{\Ga, \alpha_1 \Rightarrow \neg\alpha_1} 
\end{array}
\]
\[
\begin{array}{ccccccccc}
\infer[\copc]{\De, \neg \alpha_1 \Rightarrow \neg\gamma}{\De, \neg \alpha_1, \gamma \Rightarrow \alpha_1} 
\end{array}
\]
By the induction hypothesis twice, we have derivations
\[
\begin{array}{ccccccccc}
\infer[]{\Ga, \De, \gamma \Rightarrow \alpha_1}{\Ga \Rightarrow \neg\alpha_1 \qquad \De, \neg \alpha_1, \gamma \Rightarrow \alpha_1} 
\end{array}
\]
and
\[
\begin{array}{ccccccccc}
\infer[]{\Ga, \De, \alpha_1 \Rightarrow \neg\gamma}{\Ga, \alpha_1 \Rightarrow \neg\alpha_1 \qquad \De, \neg \alpha_1 \Rightarrow \neg\gamma} 
\end{array}
\]
By the induction hypothesis again, and the weakening rule, we also get
\[
\begin{array}{ccccccccc}
\infer[]{\Ga, \De, \gamma \Rightarrow \neg\gamma}{\Ga, \De, \gamma \Rightarrow \alpha_1 \qquad \Ga, \De, \alpha_1 \Rightarrow \neg\gamma} 
\end{array}
\]
But then, by using $\anr$, we get a derivation of $\Ga, \De\Rightarrow \neg\gamma$. 

The remaining cases go similarly, and are left to the reader---who can also find them in~\cite{Colacito:Thesis:2016}.
\end{proof}


It is immediate to see that the calculi can be used to give an effective proof that, if the sequent $\Ga \seq \varphi \lor \psi$ can be proved via a derivation of height $n \in \mathbb{N}$, and $\Ga$ does not contain any disjunction, then (at least) one of the two sequents $\Ga \seq \varphi$ and $\Ga \seq \psi$ can be proved via a derivation of height at most $n$. The most striking consequence of cut elimination is the possibility of computable proof search, and hence, of providing a proof-theoretic decision method for the considered logics. As for complexity results, we cannot say much with the tools developed in the first part of the paper. In fact, proof search in the calculi presented in Figure~\ref{fig:g} need not terminate, since we might encounter loops resulting in infinite branches. 

\begin{example}\label{ex:nonterm}
Consider a proof-search procedure for the sequent $\neg\neg\neg p\Rightarrow\neg\neg p$ in the calculus for {\sf CoPC}. The only possibility is to consider a backward application of $\copc$, obtaining the sequent $\neg\neg\neg p,\neg p\Rightarrow\neg\neg p$. At this point, considering a further backward application of $\copc$, with context $\neg p$, we get $\neg\neg\neg p,\neg p,\neg p\Rightarrow\neg\neg p$, and we can already see that the bottom-up proof-search procedure does not necessarily terminate.
\end{example}

Observe that sequents of the form $\Ga, \varphi \Rightarrow \varphi$ can be proved to be derivable for every formula $\varphi$. By this and a straightforward induction on the height of a derivation, we can also prove that we can uniformly substitute formulas in derivable sequents obtaining derivable sequents. The following result reduces the number of negated formulas to be considered when dealing with any system containing the $\copc$ rule. 

\begin{prop}\label{p:negneg}
Let $n \in \mathbb{N}$ be an arbitrary natural number such that $n \ge 1$. Then, we have that $\neg^{(2n+1)} p \to \neg p \text{ and } \neg^{(2n)} p \leftrightarrow \neg\neg p$ are derivable via the rule \copc,
where $\neg^{(m)}$ denotes $m$ nested applications of the negation operator for any natural number $m$.
\end{prop}

\begin{proof}
The proof goes by induction on $n \in \mathbb{N}$. The base case consists of presenting a derivation of the sequent $\Rightarrow \neg\neg\neg p \to \neg p$. After an application of $\implrr$, it is sufficient to apply to the obtained sequent $\neg\neg\neg p \Rightarrow \neg p$ the rule $\copc$ three times, with principal formulas respectively the pairs $(\neg\neg\neg p, \neg p)$, $(\neg\neg\neg p, \neg\neg p)$, and $(\neg p, \neg\neg p)$. This leads to the axiom:
\[
 \begin{array}{ccccccccc}
 \infer[\axr]{\neg\neg\neg p, \neg p, \neg p, p \Rightarrow p}{} & 
 \end{array}
\]

As for the induction step for $n > 1$, by the induction hypothesis and the invertibility of $\implrr$, we have derivations of the sequents $\neg^{(2n-1)} p \Rightarrow \neg p$ and $\neg\neg\neg p \Rightarrow \neg p$. From $\neg^{(2n-1)} p \Rightarrow \neg p$, by using $\weakr$ and $\copc$,

\[
\begin{array}{ccccccccc}
\infer[]{\neg\neg p \Rightarrow \neg^{(2n)} p}{\infer[]{\neg\neg p, \neg^{(2n-1)} p \Rightarrow \neg p}{\neg^{(2n-1)} p \Rightarrow \neg p}} 
\end{array}
\]
\noindent and hence, we get $\neg\neg p \to \neg^{(2n)} p$.
Further, by $\weakr$ and $\copc$, 

\[
\begin{array}{ccccccccc}
\infer[]{\neg^{(2n+1)} p \Rightarrow \neg\neg\neg p}{\infer[]{\neg^{(2n+1)} p, \neg\neg p \Rightarrow \neg^{(2n)} p}{\neg\neg p \Rightarrow \neg^{(2n)} p}} 
\end{array}
\]

\noindent 
But then, by $\cutr$, 

\[
\begin{array}{ccccccccc}
\infer[]{\neg^{(2n+1)} p \Rightarrow \neg p}{\neg^{(2n+1)} p \Rightarrow \neg\neg\neg p \qquad \neg\neg\neg p \Rightarrow \neg p} 
\end{array}
\]
\noindent to get $\neg^{(2n+1)} p \to \neg p$.
Finally, by the induction hypothesis, the formula $\neg^{(2n-1)} p \to \neg p$ is a theorem. But then, by uniformly substituting $\neg q$ for $p$ in it, we obtain the desired $\neg^{(2n)} q \to \neg\neg q$. 
\end{proof}

It is worth mentioning a further application of the calculi, whose details can be found in~\cite{Colacito:Thesis:2016,colacito2016subminimal}. If $\varphi$ is a formula in the language $\mathcal{L}({\sf Prop})$, the inductively defined $\sim$ operator that leaves the atoms unchanged, maps $\varphi \circ \psi$ to $(\varphi^{\sim} \circ \psi^{\sim})$ for $\circ \in \{\land, \lor, \to\}$, and $\neg \varphi$ to $(\varphi^{\sim} \to \neg \varphi^{\sim})$, soundly and truthfully translates {\sf MPC} into the logic {\sf CoPC} (and, even more, into the logic {\sf N}). 

We conclude this section with a remark. It can be seen by looking at the rules of the proof systems presented in Figure~\ref{fig:g} that none of the theorems of the logics {\sf N}, {\sf NeF}, and {\sf CoPC} is of the form $\neg \varphi$ (or, similarly, $\neg\varphi \lor \neg\psi$). In fact, there is no way to start the proof search for a sequent of the form $\seq \neg \varphi$, since the only rules for negation that can be applied require the left-hand side of the sequent to contain a negated formula.

\subsection{Craig's Interpolation}

Craig's interpolation theorem states that for each provable implication $\phi\to\psi$, there exists an interpolant $\sigma$ in the common language of $\psi$ and $\phi$, with both $\phi\to\sigma$ and $\sigma\to\psi$ being provable. Syntactic proofs of this result for intuitionistic logic go back to Sch\"{u}tte~\cite{SCH62}, and Maehara~\cite{M60} (cf.~\cite{troelstra2000basic}), and use the idea of splitting sequents in Gentzen's calculus. We adopt the same method here, and prove the result by proof-theoretical means for the logics we consider in this paper, using the analytical calculi we have at hand. 

By common language of two multisets of formulas $\Ga$ and $\De$ we mean the set of propositional variables $\{p_1, \ldots, p_n\}$ that appear both in (at least) a formula from $\Ga$ and in (at least) a formula from $\De$. It is not necessary to require the common language to be non-empty. In fact, if this is the case, the interpolant is given by the constant $\top$ (or by a formula obtained from $\top$ via negations).

\begin{theorem}
Let $\Ga,\De$ be finite multisets of formulas and let $\varphi$ be a formula such that the common language of $\Ga$ and $\De,\varphi$ is not empty. If $\vdash \Ga,\De \Rightarrow \varphi$, there exists a formula $\sigma$ such that:

\begin{enumerate}

\item [\rm 1.] the language of $\sigma$ is contained in the common language of $\Ga$ and $\De,\varphi$,
\item [\rm 2.] $\vdash \Ga \Rightarrow \sigma$ and $\vdash \De, \sigma\Rightarrow \varphi$.

\end{enumerate}

\end{theorem}

\begin{proof}
The proof goes by induction on the height of a derivation $\vdash \Ga,\De \Rightarrow \varphi$, where the induction hypothesis states that an interpolant $\sigma$ exists for every possible splitting of the considered sequents. Given a sequent $\Ga, \De \Rightarrow \varphi$, we use $\Ga'$ to denote a sub-multiset of $\Ga$ and $\De'$ for a sub-multiset of $\De$. We focus on the negative rules, and for the positive ones we refer to~\cite{troelstra2000basic}.

Suppose that the derivation ends with an application of $\nr$ of the form 
\[
\begin{array}{ccccccccc}
\infer[\nr]{\Ga', \neg\alpha, \De\Rightarrow \neg\beta}{\Ga', \neg\alpha, \alpha, \De \Rightarrow \beta \qquad \Ga', \neg\alpha, \De, \beta \Rightarrow \alpha} 
\end{array}
\] 
By the induction hypothesis twice, there exist $\sigma_1$ and $\sigma_2$ in the common language of $\Ga', \neg\alpha$ and $\De, \beta$ such that the sequents 
\[
\Ga', \neg\alpha, \alpha \Rightarrow \sigma_1;\,\,
\De, \sigma_1 \Rightarrow \beta;\,\,
\Ga', \neg\alpha, \sigma_2 \Rightarrow \alpha;
\text{ and }
\De, \beta \Rightarrow \sigma_2
\]
are derivable. We are going to prove that the desired interpolant is of the form: 
\[
(\sigma_1 \to \sigma_2) \to ((\sigma_2 \to \sigma_1) \land \neg\sigma_1).
\]
First of all, observe that the considered formula is in the common language of $\Ga', \neg\alpha$ and $\De, \beta$. It remains to find derivations of the sequents 
\[
\Ga', \neg\alpha \Rightarrow (\sigma_1 \to \sigma_2) \to ((\sigma_2 \to \sigma_1) \land \neg\sigma_1)
\text{ and }
\]
\[
\De, (\sigma_1 \to \sigma_2) \to ((\sigma_2 \to \sigma_1) \land \neg\sigma_1) \Rightarrow \neg\beta.
\]
By using $\cutr$ and $\cn$, 
\[
\begin{array}{ccccccccc}
\infer[]{\Ga', \neg\alpha, \sigma_2 \Rightarrow \sigma_1}{\Ga', \neg\alpha, \sigma_2 \Rightarrow \alpha \qquad \Ga', \neg\alpha, \alpha \Rightarrow \sigma_1}
\end{array}
\]
But then, by $\implrr$, we get $\Ga', \neg\alpha \Rightarrow \sigma_2 \to \sigma_1$, and we weaken it to obtain a derivation of $\Ga', \neg\alpha, \sigma_1 \to \sigma_2 \Rightarrow \sigma_2 \to \sigma_1\,\,(*)$. Now, by using $\weakr$, $\impllr$ and the derivability of the sequent $\Ga', \neg\alpha, \sigma_1 \to \sigma_2, \sigma_1 \Rightarrow \sigma_1$, we obtain
\[
\begin{array}{ccccccccc}
\infer[\impllr]{\Ga', \neg\alpha, \sigma_1, \sigma_1 \to \sigma_2 \Rightarrow \alpha}{\Ga', \neg\alpha, \sigma_1 \to \sigma_2, \sigma_1 \Rightarrow \sigma_1 \qquad \Ga', \neg\alpha, \sigma_2, \sigma_1 \Rightarrow \alpha} 
\end{array}
\]
But then, by $\weakr$ and $\nr$,
\[
\begin{array}{ccccccccc}
\infer[\nr]{\Ga', \neg\alpha, \sigma_1 \to \sigma_2 \Rightarrow \neg\sigma_1}{\Ga', \neg\alpha, \alpha, \sigma_1 \to \sigma_2 \Rightarrow \sigma_1 & & \Ga', \neg\alpha, \sigma_1, \sigma_1 \to \sigma_2 \Rightarrow \alpha} 
\end{array}
\] 
From the latter and from $(*)$, by $\conjrr$ and $\implrr$,
\[
\begin{array}{ccccccccc}
\infer[\implrr]{\Ga', \neg\alpha \Rightarrow (\sigma_1 \to \sigma_2) \to ((\sigma_2 \to \sigma_1) \land \neg\sigma_1)}{\infer[\conjrr]{\Ga', \neg\alpha, \sigma_1 \to \sigma_2 \Rightarrow ((\sigma_2 \to \sigma_1) \land \neg\sigma_1)}{\Ga', \neg\alpha, \sigma_1 \to \sigma_2 \Rightarrow \sigma_2 \to \sigma_1 \qquad \Ga', \neg\alpha, \sigma_1 \to \sigma_2 \Rightarrow \neg\sigma_1}} 
\end{array}
\] 
At this point we start again and, by $\cutr$, $\cn$ and $\implrr$, we get
\[
\begin{array}{ccccccccc}
\infer[\implrr]{\De \Rightarrow \sigma_1 \to \sigma_2}{\infer[\cutr]{\De, \sigma_1 \Rightarrow \sigma_2}{\De, \sigma_1 \Rightarrow \beta \qquad \De, \beta \Rightarrow \sigma_2}}
\end{array}
\]
Now, by $\weakr$, we obtain a derivation of 
\[
\De, (\sigma_1 \to \sigma_2) \to ((\sigma_2 \to \sigma_1) \land \neg\sigma_1) \Rightarrow \sigma_1 \to \sigma_2.
\]
We shall use the obtained sequent as the left premise of an application of $\impllr$, and hence, we look for a derivation of $\De, ((\sigma_2 \to \sigma_1) \land \neg\sigma_1) \Rightarrow \neg\beta$. By the induction hypothesis, we get the following derivations: 
\[
\begin{array}{ccccccccc}
\infer[\impllr]{\De, \sigma_2 \to \sigma_1, \neg\sigma_1, \beta \Rightarrow \sigma_1}{\De, \sigma_2 \to \sigma_1, \neg\sigma_1, \beta \Rightarrow \sigma_2 & & \De, \neg\sigma_1, \beta, \sigma_1 \Rightarrow \sigma_1}
\end{array}
\]
\[
\begin{array}{ccccccccc}
\infer[\conjlr]{\De, ((\sigma_2 \to \sigma_1) \land \neg\sigma_1) \Rightarrow \neg\beta}{\infer[\nr]{\De, \sigma_2 \to \sigma_1, \neg\sigma_1 \Rightarrow \neg\beta}{\De, \sigma_2 \to \sigma_1, \neg\sigma_1, \sigma_1 \Rightarrow \beta & & \De, \sigma_2 \to \sigma_1, \neg\sigma_1, \beta \Rightarrow \sigma_1}}
\end{array}
\]

Suppose now that the derivation ends with an application of $\nr$ of the form 
\[
\begin{array}{ccccccccc}
\infer[\nr]{\Ga, \De', \neg\alpha \Rightarrow \neg\beta}{\Ga, \De', \neg\alpha, \alpha \Rightarrow \beta \qquad \Ga, \De', \neg\alpha, \beta \Rightarrow \alpha} 
\end{array}
\] 
By the induction hypothesis twice, there exist $\sigma_1$ and $\sigma_2$ in the common language of $\Ga$ and $\De', \neg\alpha, \beta$ such that the sequents 
\[
\Ga \Rightarrow \sigma_1;\,\,
\De', \neg\alpha, \alpha, \sigma_1 \Rightarrow \beta;\,\,
\Ga \Rightarrow \sigma_2;
\text{ and }
\De', \neg\alpha, \beta, \sigma_2 \Rightarrow \alpha
\]
are derivable. By $\conjrr$, 
\[
\begin{array}{ccccccccc}
\infer[\conjrr]{\Ga \Rightarrow \sigma_1 \land \sigma_2}{\Ga \Rightarrow \sigma_1 \qquad \Ga \Rightarrow \sigma_2} 
\end{array}
\] 
Now, by $\weakr$, $\nr$ and $\conjlr$,
\[
\begin{array}{ccccccccc}
\infer[\conjlr]{\De', \neg\alpha, \sigma_1 \land \sigma_2 \Rightarrow \neg\beta}{\infer[\nr]{\De', \neg\alpha, \sigma_1, \sigma_2 \Rightarrow \neg\beta}{\De', \neg\alpha, \alpha, \sigma_1, \sigma_2 \Rightarrow \beta \qquad \De', \neg\alpha, \beta, \sigma_1, \sigma_2 \Rightarrow \alpha}} 
\end{array}
\] 
But now, given that $\sigma_1 \land \sigma_2$ is in the common language of $\Ga$ and $\De', \neg\alpha, \beta$, it is the desired interpolant.

Now, suppose that the derivation ends with an application of $\nefr$ of the form 
\[
\begin{array}{ccccccccc}
\infer[\nefr]{\Ga', \neg\alpha, \De \Rightarrow \neg\beta}{\Ga', \neg\alpha, \De \Rightarrow \alpha} 
\end{array}
\] 
By the induction hypothesis, there exists a formula $\sigma_1$ in the common language of $\Ga', \neg\alpha$ and $\De$ such that the sequents 
\[
\Ga', \neg\alpha, \sigma_1 \Rightarrow \alpha\text{ and }
\De \Rightarrow \sigma_1
\]
are derivable. By $\weakr$, $\nefr$ and $\impllr$, 
\[
\begin{array}{ccccccccc}
\infer[\impllr]{\De,\sigma_1 \to \neg\sigma_1 \Rightarrow \neg\beta}{\De, \sigma_1 \to \neg\sigma_1 \Rightarrow \sigma_1 \qquad \infer[\nefr]{\De, \neg\sigma_1 \Rightarrow \neg\beta}{\De, \neg\sigma_1 \Rightarrow \sigma_1}} 
\end{array}
\] 
But also, by $\nefr$ and $\implrr$, 
\[
\begin{array}{ccccccccc}
\infer[\implrr]{\Ga', \neg\alpha \Rightarrow \sigma_1 \to \neg\sigma_1}{\infer[\nefr]{\Ga', \neg\alpha, \sigma_1 \Rightarrow \neg\sigma_1}{\Ga', \neg\alpha, \sigma_1 \Rightarrow \alpha}} 
\end{array}
\] 
Now it is sufficient to notice that the formula $\sigma_1 \to \neg\sigma_1$ is in fact in the common language of $\Ga', \neg\alpha$ and $\De$ and hence, is the desired interpolant.

At this point, suppose that the derivation ends with an application of $\nefr$ of the form
\[
\begin{array}{ccccccccc}
\infer[\nefr]{\Ga, \De', \neg\alpha \Rightarrow \neg\beta}{\Ga, \De', \neg\alpha \Rightarrow \alpha} 
\end{array}
\] 
By the induction hypothesis, there exists $\sigma_1$ in the common language of $\Ga$ and $\De', \neg\alpha$ such that the sequents 
\[
\Ga \Rightarrow \sigma_1 \text{ and }
\De', \neg\alpha, \sigma_1 \Rightarrow \alpha
\]
are derivable. By a straightforward application of $\nefr$, we get a derivation of $\De', \neg\alpha, \sigma_1 \Rightarrow \neg\beta$, and we conclude that $\sigma_1$ is the desired interpolant.

Suppose that the derivation ends with an application of $\copc$ of the form 
\[
\begin{array}{ccccccccc}
\infer[\copc]{\Ga', \neg\alpha, \De \Rightarrow \neg\beta}{\Ga', \neg\alpha, \beta, \De \Rightarrow \alpha} 
\end{array}
\] 
By the induction hypothesis, there exists $\sigma_1$ in the common language of $\Ga', \neg\alpha$ and $\De, \beta$ such that the sequents 
\[
\Ga', \neg\alpha, \sigma_1 \Rightarrow \alpha,
\text{ and }
\De, \beta \Rightarrow \sigma_1
\]
are derivable. By using $\copc$ on the first sequent, we get a derivation of $\Ga', \neg\alpha \Rightarrow \neg\sigma_1$. Moreover, by $\weakr$ and $\copc$ we derive $\De, \neg\sigma_1 \Rightarrow \neg\beta$.

Suppose now that the derivation ends with an application of $\copc$ of the form 
\[
\begin{array}{ccccccccc}
\infer[\copc]{\Ga, \De', \neg\alpha \Rightarrow \neg\beta}{\Ga, \De', \neg\alpha, \beta \Rightarrow \alpha} 
\end{array}
\] 
By the induction hypothesis, there exists $\sigma_1$ in the common language of $\Ga$ and $\De', \neg\alpha, \beta$ such that the sequents 
\[
\Ga \Rightarrow \sigma_1,
\text{ and }
\De', \neg\alpha, \beta, \sigma_1 \Rightarrow \alpha
\]
are derivable. By $\copc$, we get a derivation of $\De', \neg\alpha, \sigma_1 \Rightarrow \neg\beta$. Hence, $\sigma_1$ is the desired interpolant.

Finally, suppose that the derivation ends with an application of $\anr$ 
\[
\begin{array}{ccccccccc}
\infer[\anr]{\Ga, \De \Rightarrow \neg\alpha}{\Ga, \De, \alpha \Rightarrow \neg\alpha} 
\end{array}
\] 
By the induction hypothesis, there is a formula $\sigma_1$ in the common language of $\Ga$ and $\De, \neg\alpha$ such that the sequents 
\[
\Ga \Rightarrow \sigma_1,
\text{ and }
\De, \alpha, \sigma_1 \Rightarrow \neg\alpha
\]
are derivable. By using $\anr$ on the second sequent, we get a derivation of $\De, \sigma_1 \Rightarrow \neg\alpha$ and hence, the desired interpolant is exactly $\sigma_1$.
\end{proof}


\section{Terminating Systems}\label{s:4}

For the systems considered in the previous section, a naive backward proof search strategy is clearly not terminating (cf.\ Example~\ref{ex:nonterm}).
To obtain terminating systems, we adapt a \emph{history mechanism} developed in~\cite{heuerding1996efficient} and~\cite{howe1998proof}. We mention that the method employed here is obtained by modifying what in~\cite{howe1998proof} is called `Swiss history' (or `Bern Approach'). For further literature, see, e.g.,~\cite{heuerding1996efficient,Heu98}. Another variant is given by the so-called `Scottish history' (or `St.\ Andrews Approach')~\cite{How96,How97}. 

The original idea of preventing loops in a proof search using history mechanisms consists in adding a {\em history} to a sequent in order to store information about all the sequents that have occurred so far on a branch of a proof search tree. The method employed here is based on the assumption that storing {goal} formulas---given a sequent $\Ga \Rightarrow \varphi$, we refer to the formula $\varphi$ as the \emph{goal} of the sequent---is sufficient to avoid possible loops.

We consider modifications of the $\m{G3}$ sequent calculi presented in Section~\ref{s:2}. We use $\m{G3}^{\varphi}$ to denote the modification of the $\m{G3}$ calculi obtained by considering the goals of the left rules to be restricted to atoms, negations or disjunctions. Here we call a left (respectively, right) rule an inference rule which introduces a principal formula on the left (respectively, right). 

Further, we employ a history mechanism, and obtain the sequent calculi $\m{G3}^{Hist}$ from Figures~\ref{fig:ghist1} and~\ref{fig:ghist2}. The goal formula $\varphi$ in $\m{G3}^{Hist}$ is again either an atom, a negation or a disjunction. Note that reading the $\m{G3}^{Hist}$ rules bottom-up, the context is always non-decreasing; also, the left rules, the rule $\implrr$, and the rule $\anr$ have side conditions for the context and/or for the history. The side conditions make sure that we empty the history set only if the context is {\em properly extended}. If this is not the case, we keep the history as it is and, if necessary, we enhance it with the new goal formula. If we meet a sequent whose goal is already in the history, then we know that we are dealing with a loop. Therefore, having a derivation of the sequent $\mh \h \Ga \Rightarrow \varphi$ in one of the $\m{G3}^{Hist}$ systems amounts to saying that the proof tree of $\mh \h \Ga \Rightarrow \varphi$ does not contain sequents of the form $\mh' \h \Ga \Rightarrow \psi$ for every $\psi \in \mh$. Note that the rule $\conjlr$ in the terminating system is not purely $\m{G3}$, but is a $\m{G2}$-type rule with Kleene's trick of building contraction in. We define the size of a proof tree as the number of nodes of the tree. 

\begin{figure}[h]
\begin{mdframed}[style=mdfexample1]
\[
 \begin{array}{ccccccccc}
 \infer[\axr]{\mh \h \Ga, p \Rightarrow p}{} & 
 \qquad &
 \infer[{\topr}]{\mh \h \Ga \Rightarrow \top}{} &
 \end{array}
\]
\[
 \begin{array}{ccccccccc}
 \infer[\implrrone]{\mh \h \Ga \Rightarrow \alpha\to\beta}{\emptyset \h \Ga, \alpha \Rightarrow \beta} 
 \qquad &
 \infer[\implrrtwo]{\mh \h \Ga \Rightarrow \alpha\to\beta}{\mh \h \Ga \Rightarrow \beta} 
 \smallskip \\
 \mbox{if $\alpha \not \in \Ga$}
 \qquad &
 \mbox{if $\alpha \in \Ga$}
 \end{array}
\]
\[
 \begin{array}{ccccccccc}
 \infer[\impllr]{\mh \h \Ga, \alpha\to\beta \Rightarrow \varphi}{(\varphi,\mh) \h \Ga, \alpha\to\beta \Rightarrow \alpha \qquad \emptyset \h \Ga, \alpha\to\beta, \beta \Rightarrow \varphi} 
 \smallskip \\ 
 \mbox{if $\varphi \not \in \mh$ and $\beta \not \in \Ga$}
 \end{array}
\]
\[
 \begin{array}{ccccccccc}
 \infer[\conjrr]{\mh \h \Ga \Rightarrow \alpha\land\beta}{\mh \h \Ga \Rightarrow \alpha \qquad \mh \h \Ga \Rightarrow \beta} 
 \end{array}
\]
\[
 \begin{array}{ccccccccc}
 \infer[\conjlrone]{\mh \h \Ga, \alpha\land\beta \Rightarrow \varphi}{\emptyset \h \Ga, \alpha\land\beta, \alpha \Rightarrow \varphi} 
 \qquad &
 \infer[\conjlrtwo]{\mh \h \Ga, \alpha\land\beta \Rightarrow \varphi}{\emptyset \h \Ga, \alpha\land\beta, \beta \Rightarrow \varphi} 
 \smallskip \\
 \mbox{if $\alpha \not \in \Ga$}
 \qquad &
 \mbox{if $\beta \not \in \Ga$}
 \end{array}
\]
\[
 \begin{array}{ccccccccc}
 \infer[\disjrrone]{\mh \h \Ga \Rightarrow \alpha\lor\beta}{\mh \h \Ga \Rightarrow \alpha} 
 \qquad &
 \infer[\disjrrtwo]{\mh \h \Ga \Rightarrow \alpha\lor\beta}{\mh \h \Ga \Rightarrow \beta} 
 \end{array}
\]
\[
 \begin{array}{ccccccccc}
 \infer[\disjlr]{\mh \h \Ga, \alpha\lor\beta \Rightarrow \varphi}{\emptyset \h \Ga, \alpha\lor\beta, \alpha \Rightarrow \varphi \qquad \emptyset \h \Ga, \alpha\lor\beta, \beta \Rightarrow \varphi} 
 \smallskip \\ 
 \mbox{if $\alpha, \beta \not \in \Ga$}
 \end{array}
\]
\end{mdframed}
\caption{Rules for positive connectives in the systems $\m{G3}^{Hist}$} \label{fig:ghist1}
\end{figure}

\begin{figure}[h]
\begin{mdframed}[style=mdfexample1]
\[
 \begin{array}{ccccccccc}
 \infer[\none]{\mh \h \Ga, \neg\alpha \Rightarrow \neg\beta}{\emptyset \h \Ga, \neg\alpha, \beta \Rightarrow \alpha \qquad \emptyset \h \Ga, \neg\alpha, \alpha \Rightarrow \beta} 
 \smallskip \\
 \mbox{if $\beta \not \in \Ga\cup\{\neg\alpha\}$ and $\alpha \not\in \Ga$}
 \end{array}
\]
 \[
 \begin{array}{ccccccccc}
 \infer[\ntwo]{\mh \h \Ga, \neg\alpha \Rightarrow \neg\beta}{\emptyset \h \Ga, \neg\alpha, \beta \Rightarrow \alpha \qquad \mh \h \Ga, \neg\alpha \Rightarrow \beta} 
 \smallskip \\
 \mbox{if $\beta \not \in \Ga\cup\{\neg\alpha\}$ and $\alpha \in \Ga$}
 \end{array}
\]
\[
 \begin{array}{ccccccccc}
 \infer[\nthree]{\mh \h \Ga, \neg\alpha \Rightarrow \neg\beta}{(\neg\beta,\mh) \h \Ga, \neg\alpha \Rightarrow \alpha \qquad \emptyset \h \Ga, \neg\alpha, \alpha \Rightarrow \beta} 
 \smallskip \\
 \mbox{if $\neg\beta \not \in \mh$, $\beta \in \Ga\cup\{\neg\alpha\}$ and $\alpha \not\in \Ga$}
 \end{array}
 \]
\[
 \begin{array}{ccccccccc}
 \infer[\nfour]{\mh \h \Ga, \neg\alpha \Rightarrow \neg\beta}{(\neg\beta,\mh) \h \Ga, \neg\alpha \Rightarrow \alpha \qquad \mh \h \Ga, \neg\alpha \Rightarrow \beta} 
 \smallskip \\
 \mbox{if $\neg\beta \not \in \mh$, $\beta \in \Ga\cup\{\neg\alpha\}$ and $\alpha \in \Ga$}
 \end{array}
\]
\[
 \begin{array}{ccccccccc}
 \infer[\nefr]{\mh \h \Ga, \neg\alpha \Rightarrow \neg\beta}{(\neg\beta, \mh) \h \Ga, \neg\alpha \Rightarrow \alpha} 
 \qquad \mbox{if $\neg\beta \not \in \mh$}
 \end{array}
\]
\[
 \begin{array}{ccccccccc}
 \infer[\copcone]{\mh \h \Ga, \neg\alpha \Rightarrow \neg\beta}{\emptyset \h \Ga, \neg\alpha, \beta \Rightarrow \alpha} 
 \qquad &
 \infer[\copctwo]{\mh \h \Ga, \neg\alpha \Rightarrow \neg\beta}{(\neg\beta,\mh) \h \Ga, \neg\alpha \Rightarrow \alpha} 
 \smallskip \\
 \mbox{if $\beta \not \in \Ga\cup\{\neg\alpha\}$}
 \qquad &
 \mbox{if $\neg\beta \not \in \mh$ and $\beta \in \Ga\cup\{\neg\alpha\}$}
 \end{array}
\]
\[
 \begin{array}{ccccccccc}
 \infer[\anr]{\mh \h \Ga \Rightarrow \neg\alpha}{\emptyset \h \Ga, \alpha \Rightarrow \neg\alpha} 
 \qquad \mbox{if $\alpha \not \in \Ga$}
 \end{array}
\]
\end{mdframed}
\caption{Rules for negation in the systems $\m{G3}^{Hist}$} \label{fig:ghist2}
\end{figure}

The first step towards a proof of the equivalence between $\m{G3}$ and $\m{G3}^{Hist}$ is the following.

\begin{lemma}
The weakening rule is admissible in $\m{G3}^{\varphi}$. 
\end{lemma}

\begin{proof}
Easy consequence of~\cite[Lemma 4.1]{howe1998proof} and of Proposition~\ref{prop:w}.
\end{proof}

The following result ensures that we can restrict ourselves to $\m{G3}^{\varphi}$.

\begin{lemma}\label{lemma:eq1}
The sequent calculi $\m{G3}$ and $\m{G3}^{\varphi}$ are equivalent. 
\end{lemma}

\begin{proof}
It is trivial that if a sequent is provable in $\m{G3}^{\varphi}$, then is provable in $\m{G3}$. For the converse, the proof proceeds by induction on the height of a derivation. 

Consider a $\m{G3}$ inference which is not a $\m{G3}^{\varphi}$ inference. This must be an instance of a left rule with an implicational or a conjunctive goal. By the induction hypothesis, we have $\m{G3}^{\varphi}$ derivations of the premises. This implies the implicational or conjunctive goals to be principal formulas in the premises. Hence, the result follows from the proof of~\cite[Proposition~4.1]{howe1998proof}.
\end{proof}

\begin{lemma}
The following contraction rule is admissible in $\m{G3}^{Hist}$
\[
 \begin{array}{ccccccccc}
 \infer{\mh \h \Ga, \alpha \Rightarrow \varphi}{\mh \h \Ga, \alpha, \alpha \Rightarrow \varphi} 
 \end{array}
\]
\end{lemma}

\begin{proof}
By induction on the height of a derivation of the premise. We focus on the induction step, and on the non-intuitionistic rules. 

Suppose that the last step in the derivation is an instance of $\none$ of the form 
\[
 \begin{array}{ccccccccc}
 \infer[\none]{\mh \h \Ga, \neg\alpha, \neg\alpha \Rightarrow \neg\beta}{\emptyset \h \Ga, \neg\alpha, \neg\alpha, \beta \Rightarrow \alpha \qquad \emptyset \h \Ga, \neg\alpha, \neg\alpha, \alpha \Rightarrow \beta}
 \end{array}
\]
with $\beta \not\in \Ga\cup\{\neg\alpha\}$ and $\alpha \not\in \Ga$. By the induction hypothesis twice, 
\[
\vdash \emptyset \h \Ga, \neg\alpha, \beta \Rightarrow \alpha\text{ and }\emptyset \h \Ga, \neg\alpha, \alpha \Rightarrow \beta. 
\]
But then, by $\none$, 
\[
\vdash \mh \h \Ga, \neg\alpha \Rightarrow \neg\beta.
\]
Suppose that the last step in the derivation is an instance of $\ntwo$ of the form 
\[
 \begin{array}{ccccccccc}
 \infer[\ntwo]{\mh \h \Ga, \neg\alpha, \neg\alpha \Rightarrow \neg\beta}{\emptyset \h \Ga, \neg\alpha, \neg\alpha, \beta \Rightarrow \alpha \qquad \mh \h \Ga, \neg\alpha, \neg\alpha \Rightarrow \beta}
 \end{array}
\]
with $\beta \not\in \Ga\cup\{\neg\alpha\}$ and $\alpha \in \Ga$. By the induction hypothesis twice, 
\[
\vdash \emptyset \h \Ga, \neg\alpha, \beta \Rightarrow \alpha\text{ and }\mh \h \Ga, \neg\alpha \Rightarrow \beta. 
\]
But then, by $\ntwo$, 
\[
\vdash \mh \h \Ga, \neg\alpha \Rightarrow \neg\beta.
\]
Suppose the last step in the derivation to be an instance of $\nthree$ as follows 
\[
 \begin{array}{ccccccccc}
 \infer[\nthree]{\mh \h \Ga, \neg\alpha, \neg\alpha \Rightarrow \neg\beta}{(\neg\beta, \mh) \h \Ga, \neg\alpha, \neg\alpha \Rightarrow \alpha \qquad \emptyset \h \Ga, \neg\alpha, \neg\alpha, \alpha \Rightarrow \beta}
 \end{array}
\]
with $\beta \in \Ga\cup\{\neg\alpha\}$, $\neg\beta \not\in \mh$ and $\alpha \not\in \Ga$. By the induction hypothesis twice, we get 
\[
\vdash (\neg\beta, \mh) \h \Ga, \neg\alpha \Rightarrow \alpha\text{ and }\emptyset \h \Ga, \neg\alpha, \alpha \Rightarrow \beta. 
\]
Hence, via an application of $\nthree$, 
\[
\vdash \mh \h \Ga, \neg\alpha \Rightarrow \neg\beta.
\]
Suppose the last step in the derivation to be an instance of $\nfour$ as follows 
\[
 \begin{array}{ccccccccc}
 \infer[\nfour]{\mh \h \Ga, \neg\alpha, \neg\alpha \Rightarrow \neg\beta}{(\neg\beta, \mh) \h \Ga, \neg\alpha, \neg\alpha \Rightarrow \alpha \qquad \mh \h \Ga, \neg\alpha, \neg\alpha \Rightarrow \beta}
 \end{array}
\]
with $\beta \in \Ga\cup\{\neg\alpha\}$, $\neg\beta \not\in \mh$ and $\alpha \in \Ga$. By the induction hypothesis twice, we get 
\[
\vdash (\neg\beta, \mh) \h \Ga, \neg\alpha \Rightarrow \alpha\text{ and }\mh \h \Ga, \neg\alpha \Rightarrow \beta. 
\]
Hence, via an application of $\nfour$, 
\[
\vdash \mh \h \Ga, \neg\alpha \Rightarrow \neg\beta.
\]
Now, suppose that the last step in the derivation is an instance of $\none$ of the form 
\[
 \begin{array}{ccccccccc}
 \infer[\none]{\mh \h \Ga, \alpha, \alpha, \neg\gamma \Rightarrow \neg\beta}{\emptyset \h \Ga, \alpha, \alpha, \neg\gamma, \beta \Rightarrow \gamma \qquad \emptyset \h \Ga, \alpha, \alpha, \neg\gamma, \gamma \Rightarrow \beta}
 \end{array}
\]
where $\beta \not\in \Ga\cup\{\alpha, \neg\gamma\}$ and $\gamma \not\in \Ga\cup\{\alpha\}$. By applying the induction hypothesis twice, 
\[
\vdash \emptyset \h \Ga, \alpha, \neg\gamma, \beta \Rightarrow \gamma \text{ and } \vdash \emptyset \h \Ga, \alpha, \neg\gamma, \gamma \Rightarrow \beta.
\]
Therefore, 
\[
\vdash \mh \h \Ga, \alpha, \neg\gamma \Rightarrow \neg\beta.
\]
Now, suppose that the last step in the derivation is an instance of $\ntwo$ of the form 
\[
 \begin{array}{ccccccccc}
 \infer[\ntwo]{\mh \h \Ga, \alpha, \alpha, \neg\gamma \Rightarrow \neg\beta}{\emptyset \h \Ga, \alpha, \alpha, \neg\gamma, \beta \Rightarrow \gamma \qquad \mh \h \Ga, \alpha, \alpha, \neg\gamma \Rightarrow \beta}
 \end{array}
\]
where $\beta \not\in \Ga\cup\{\alpha, \neg\gamma\}$ and $\gamma \in \Ga\cup\{\alpha\}$. By applying the induction hypothesis twice, 
\[
\vdash \emptyset \h \Ga, \alpha, \neg\gamma, \beta \Rightarrow \gamma \text{ and } \vdash \mh \h \Ga, \alpha, \neg\gamma \Rightarrow \beta.
\]
Therefore, 
\[
\vdash \mh \h \Ga, \alpha, \neg\gamma \Rightarrow \neg\beta.
\]
Suppose now the last step in the derivation to be an instance of $\nthree$, 
\[
 \begin{array}{ccccccccc}
 \infer[\nthree]{\mh \h \Ga, \alpha, \alpha, \neg\gamma \Rightarrow \neg\beta}{(\neg\beta, \mh) \h \Ga, \alpha, \alpha, \neg\gamma \Rightarrow \gamma \qquad \emptyset \h \Ga, \alpha, \alpha, \neg\gamma, \gamma \Rightarrow \beta}
 \end{array}
\]
with $\beta \in \Ga\cup\{\alpha, \neg\gamma\}$, $\neg\beta \not\in \mh$ and $\gamma \not\in \Ga\cup\{\alpha\}$. By the induction hypothesis twice, we get 
\[
\vdash (\neg\beta, \mh) \h \Ga, \alpha, \neg\gamma \Rightarrow \gamma \text{ and } \vdash \emptyset \h \Ga, \alpha, \neg\gamma, \gamma \Rightarrow \beta.
\]
Hence, an application of $\nthree$ leads to the desired conclusion.
Suppose now the last step in the derivation to be an instance of $\nfour$, 
\[
 \begin{array}{ccccccccc}
 \infer[\nfour]{\mh \h \Ga, \alpha, \alpha, \neg\gamma \Rightarrow \neg\beta}{(\neg\beta, \mh) \h \Ga, \alpha, \alpha, \neg\gamma \Rightarrow \gamma \qquad \mh \h \Ga, \alpha, \alpha, \neg\gamma \Rightarrow \beta}
 \end{array}
\]
with $\beta \in \Ga\cup\{\alpha, \neg\gamma\}$, $\neg\beta \not\in \mh$ and $\gamma \in \Ga\cup\{\alpha\}$. By the induction hypothesis twice, we get 
\[
\vdash (\neg\beta, \mh) \h \Ga, \alpha, \neg\gamma \Rightarrow \gamma \text{ and } \vdash \mh \h \Ga, \alpha, \neg\gamma \Rightarrow \beta.
\]
Hence, an application of $\nfour$ leads to the desired conclusion.

Suppose that the last step in the derivation is an application of $\nefr$ of the form
\[
 \begin{array}{ccccccccc}
 \infer[\nefr]{\mh \h \Ga, \neg\alpha, \neg\alpha \Rightarrow \neg\beta}{(\neg\beta, \mh) \h \Ga, \neg\alpha, \neg\alpha \Rightarrow \alpha} 
 \end{array}
\]
with $\neg\beta \not\in \mh$. By the induction hypothesis, 
\[
\vdash (\neg\beta, \mh) \h \Ga, \neg\alpha \Rightarrow \alpha,
\]
and hence, by $\nefr$,
\[
\vdash \mh \h \Ga, \neg\alpha \Rightarrow \neg\beta.
\]
Now, suppose that the last step in the derivation is an instance of $\nefr$ as follows 
\[
 \begin{array}{ccccccccc}
 \infer[\nefr]{\mh \h \Ga, \alpha, \alpha, \neg\gamma \Rightarrow \neg\beta}{(\neg\beta, \mh) \h \Ga, \alpha, \alpha, \neg\gamma \Rightarrow \gamma} 
 \end{array}
\]
where $\neg\beta \not\in \mh$. Again, by the induction hypothesis, 
\[
\vdash (\neg\beta, \mh) \h \Ga, \alpha, \neg\gamma \Rightarrow \gamma.
\]
Hence, an appropriate application of $\nefr$ gives us the desired result.

Suppose that the last step in the derivation is an instance of $\copcone$ of the form 
\[
 \begin{array}{ccccccccc}
 \infer[\copcone]{\mh \h \Ga, \neg\alpha, \neg\alpha \Rightarrow \neg\beta}{\emptyset \h \Ga, \neg\alpha, \neg\alpha, \beta \Rightarrow \alpha}
 \end{array}
\]
with $\beta \not \in \Ga\cup\{\neg\alpha\}$. By the induction hypothesis, 
\[
 \vdash \emptyset \h \Ga, \neg\alpha, \beta \Rightarrow \alpha.
\]
But then, 
\[
 \begin{array}{ccccccccc}
 \infer[\copcone]{\mh \h \Ga, \neg\alpha, \Rightarrow \neg\beta}{\emptyset \h \Ga, \neg\alpha, \beta \Rightarrow \alpha}
 \end{array}
\]
Suppose now that the last step in the derivation is an instance of $\copctwo$ of the form 
\[
 \begin{array}{ccccccccc}
 \infer[\copctwo]{\mh \h \Ga, \neg\alpha, \neg\alpha \Rightarrow \neg\beta}{(\neg\beta, \mh) \h \Ga, \neg\alpha, \neg\alpha \Rightarrow \alpha}
 \end{array}
\]
with $\beta \in \Ga\cup\{\neg\alpha\}$ and $\neg\beta \not \in \mh$. By the induction hypothesis, 
\[
\vdash (\neg\beta, \mh) \h \Ga, \neg\alpha \Rightarrow \alpha
\]
and hence, by $\copctwo$, we can conclude 
\[
\vdash \mh \h \Ga, \neg\alpha \Rightarrow \neg\beta.
\]
Now suppose the last step in the derivation to be an instance of $\copcone$ of the form 
\[
 \begin{array}{ccccccccc}
 \infer[\copcone]{\mh \h \Ga, \alpha, \alpha, \neg\gamma \Rightarrow \neg\beta}{\emptyset \h \Ga, \alpha, \alpha, \neg\gamma, \beta \Rightarrow \gamma}
 \end{array}
\]
with $\beta \not \in \Ga\cup\{\alpha,\neg\gamma\}$. By the induction hypothesis, we get 
\[
\vdash \emptyset \h \Ga, \alpha, \neg\gamma, \beta \Rightarrow \gamma.
\]
But then, we can conclude 
\[
\mh \h \Ga, \alpha, \neg\gamma \Rightarrow \neg\beta
\]
Now, suppose that the last step in the derivation is an instance of $\copctwo$ of the form 
\[
 \begin{array}{ccccccccc}
 \infer[\copctwo]{\mh \h \Ga, \alpha, \alpha, \neg\gamma \Rightarrow \neg\beta}{(\neg\beta, \mh) \h \Ga, \alpha, \alpha, \neg\gamma \Rightarrow \gamma}
 \end{array}
\]
with $\beta \in \Ga\cup\{\alpha,\neg\gamma\}$ and $\neg\beta \not \in \mh$. By the induction hypothesis, 
\[
\vdash (\neg\beta, \mh) \h \Ga, \alpha, \neg\gamma \Rightarrow \gamma.
\]
Hence, by $\copctwo$, 
\[
\vdash \mh \h \Ga, \alpha, \neg\gamma \Rightarrow \neg\beta.
\]

Finally, suppose the last step in the derivation to be an instance of $\anr$ of the form 
\[
 \begin{array}{ccccccccc}
 \infer[\anr]{\mh \h \Ga, \alpha, \alpha \Rightarrow \neg\beta}{\emptyset \h \Ga, \alpha, \alpha, \beta \Rightarrow \neg\beta} 
 \end{array}
\]
where $\beta \not\in \Ga\cup\{\alpha\}$. By the induction hypothesis, 
\[
\vdash \emptyset \h \Ga, \alpha, \beta \Rightarrow \neg\beta.
\]
Hence, by $\anr$,
\[
\vdash \mh \h \Ga, \alpha \Rightarrow \neg\beta.
\]

\end{proof}

For the next result, we start from a $\m{G3}^{\varphi}$ derivation of a sequent $\Ga \Rightarrow \varphi$, and construct a proof tree in $\m{G3}^{Hist}$ for the corresponding sequent $\emptyset \h \Ga \Rightarrow \varphi$. The proof trees that we are going to use are {\em hybrid}, in the sense that we deal with fragments of $\m{G3}^{Hist}$ proof trees combined with $\m{G3}^{\varphi}$ proof trees. Namely, every branch of the $\m{G3}^{Hist}$ proof tree that do not have axiom leaves ends with a $\m{G3}^{\varphi}$ proof tree. We work from the root up, and hence we focus on a uppermost history sequent $\mh \h \Ga \Rightarrow \varphi$ with non-history premises. More precisely, the argument below shows how the history mechanism allows us to detect and remove loops. The most interesting cases are the ones concerning rules with side conditions for the history. If the history side condition is not satisfied (i.e., $\varphi \in \mh$), it means that we are dealing with a (non-trivial) loop: the same sequent $\mh' \h \Ga \Rightarrow \varphi$ already occurs previously in the proof tree (with $\mh' \subseteq \mh$). This has the further consequence that the history has not been reset at any point in this fragment of the tree, because the context is stable (i.e., has not been properly extended). At this point, we know how to remove the loop, by removing the fragment of the tree from, but not including, the sequent $\mh' \h \Ga \Rightarrow \varphi$ up to and including $\mh \h \Ga \Rightarrow \varphi$. 

\begin{theorem}
The calculi $\m{G3}$ and $\m{G3}^{Hist}$ are equivalent. That is, a sequent $\Ga \Rightarrow \varphi$ is provable in $\m{G3}$ if, and only if, the sequent $\emptyset \h \Ga \Rightarrow \varphi$ is provable in $\m{G3}^{Hist}$.
\end{theorem}

\begin{proof}
From Lemma \ref{lemma:eq1} it is enough to show that $\m{G3}^{Hist}$ is equivalent to $\m{G3}^{\varphi}$. It is immediate that any sequent provable in $\m{G3}^{Hist}$ is also provable in $\m{G3}^{\varphi}$, just by dropping the history and, if necessary, using some contractions. For the other direction, we enhance the proof of~\cite[Theorem 4.1]{howe1998proof} by adding arguments for the negation rules where needed.

Suppose that the last inference rule applied is $\impllr$. If the side conditions are satisfied, then we follow the proof in~\cite{howe1998proof}. If not, we have 
\[
 \begin{array}{ccccccccc}
 \infer[\impllr]{\mh \h \Ga, \alpha\to\beta \Rightarrow \varphi}{\Ga, \alpha\to\beta \Rightarrow \alpha \qquad \Ga, \alpha\to\beta, \beta \Rightarrow \varphi} 
 \end{array}
\]
The only tricky case is when $\varphi \in \mh$. If the history condition is not met, we have to consider the additional cases in which $\varphi$ is $\neg\delta$. 

Suppose that below the conclusion the hybrid tree has the form: 
\[
 \begin{array}{ccccccccc}
 \infer[\impllr]{\mh \h \Ga, \alpha\to\beta \Rightarrow \neg\delta}{\Ga, \alpha\to\beta \Rightarrow \alpha \qquad \Ga, \alpha\to\beta, \beta \Rightarrow \neg\delta}
 \\
 \vdots
 \smallskip \\
 \infer[\nthree]{\mh' \h \Ga', \alpha\to\beta, \neg\gamma \Rightarrow \neg\delta}{(\neg\delta, \mh') \h \Ga', \alpha\to\beta, \neg\gamma \Rightarrow \gamma \qquad \ldots}
 \end{array}
\]
where $\neg\delta \not\in \mh'\subseteq\mh$ and $\Ga = \Ga'\cup\{\neg\gamma\}$. The new hybrid tree is obtained by removing from the hybrid tree described above all the sequents from, and not including, $\mh' \h \Ga', \alpha\to\beta, \neg\gamma \Rightarrow \neg\delta$ up to, and including, $\mh \h \Ga, \alpha\to\beta \Rightarrow \neg\delta$. We can now apply backwards $\impllr$ to the first of these sequents (we may need some contractions). 

Suppose that below the conclusion the hybrid tree has the form: 

\[
 \begin{array}{ccccccccc}
 \infer[\impllr]{\mh \h \Ga, \alpha\to\beta \Rightarrow \neg\delta}{\Ga, \alpha\to\beta \Rightarrow \alpha \qquad \Ga, \alpha\to\beta, \beta \Rightarrow \neg\delta}
 \\
 \vdots
 \smallskip \\
 \infer[\nfour]{\mh' \h \Ga', \alpha\to\beta, \neg\gamma \Rightarrow \neg\delta}{(\neg\delta, \mh') \h \Ga', \alpha\to\beta, \neg\gamma \Rightarrow \gamma \qquad \ldots}
 \end{array}
\]
where $\neg\delta \not\in \mh'\subseteq\mh$ and $\Ga = \Ga'\cup\{\neg\gamma\}$. The new hybrid tree is obtained by removing from the hybrid tree described above all the sequents from, and not including, $\mh' \h \Ga', \alpha\to\beta, \neg\gamma \Rightarrow \neg\delta$ up to, and including, $\mh \h \Ga, \alpha\to\beta \Rightarrow \neg\delta$. We can now apply backwards $\impllr$ to the first of these sequents (we may need some contractions).

Suppose now that the hybrid tree has the form:

\[
 \begin{array}{ccccccccc}
 \infer[\impllr]{\mh \h \Ga, \alpha\to\beta \Rightarrow \neg\delta}{\Ga, \alpha\to\beta \Rightarrow \alpha \qquad \Ga, \alpha\to\beta, \beta \Rightarrow \neg\delta}
 \\
 \vdots
 \smallskip \\
 \infer[\copctwo]{\mh' \h \Ga', \alpha\to\beta, \neg\gamma \Rightarrow \neg\delta}{(\neg\delta, \mh') \h \Ga', \alpha\to\beta, \neg\gamma \Rightarrow \gamma}
 \end{array}
\]

\noindent where $\neg\delta \not\in \mh' \subseteq \mh$ and $\Ga = \Ga'\cup\{\neg\gamma\}$. The new hybrid tree is obtained by removing from the hybrid tree described above all the sequents from, and not including, $\mh' \h \Ga', \alpha\to\beta, \neg\gamma \Rightarrow \neg\delta$ up to, and including, $\mh \h \Ga, \alpha\to\beta \Rightarrow \neg\delta$. We can now apply backwards $\impllr$ to the first of these sequents (we may need some contractions).

Finally, suppose that the hybrid tree has the following form:
\[
 \begin{array}{ccccccccc}
 \infer[\impllr]{\mh \h \Ga, \alpha\to\beta \Rightarrow \neg\delta}{\Ga, \alpha\to\beta \Rightarrow \alpha \qquad \Ga, \alpha\to\beta, \beta \Rightarrow \neg\delta}
 \\
 \vdots
 \smallskip \\
 \infer[\nefr]{\mh' \h \Ga', \alpha\to\beta, \neg\gamma \Rightarrow \neg\delta}{(\neg\delta, \mh') \h \Ga', \alpha\to\beta, \neg\gamma \Rightarrow \gamma}
 \end{array}
\]
where $\neg\delta \not\in \mh' \subseteq \mh$ and $\Ga = \Ga' \cup \{\neg\gamma\}$. The new hybrid tree is obtained by removing from the hybrid tree described above all the sequents from, and not including, $\mh' \h \Ga', \alpha\to\beta, \neg\gamma \Rightarrow \neg\delta$ up to, and including, $\mh \h \Ga, \alpha\to\beta \Rightarrow \neg\delta$. We can now apply backwards $\impllr$ to the first of these sequents (we may need some contractions).

Suppose now that the last step in the derivation is an application of $\nr$. If one among the side conditions of the history rules $\none, \ntwo, \nthree, \nfour$ is satisfied, we simply add the appropriate history to the premises and possibly apply contraction. Otherwise, we have the following situation:
\[
 \begin{array}{ccccccccc}
 \vdots
 \smallskip \\
 \infer[\nr]{\mh \h \Ga, \neg\alpha \Rightarrow \neg\beta}{\Ga, \neg\alpha, \beta \Rightarrow \alpha \qquad \Ga, \neg\alpha, \alpha \Rightarrow \beta}
 \end{array}
\]
with $\beta\in\Ga\cup\{\neg\alpha\}$ and $\neg\beta\in\mh$. Suppose that below the conclusion the hybrid tree has the form 

\[
 \begin{array}{ccccccccc}
 \infer[\nr]{\mh \h \Ga, \neg\alpha \Rightarrow \neg\beta}{\Ga, \neg\alpha, \beta \Rightarrow \alpha \qquad \Ga, \neg\alpha, \alpha \Rightarrow \beta}
 \\
 \vdots
 \smallskip \\
 \infer[\impllr]{\mh' \h \Ga', \gamma\to\delta, \neg\alpha \Rightarrow \neg\beta}{(\neg\beta, \mh') \h \Ga', \gamma\to\delta, \neg\alpha \Rightarrow \gamma \qquad \ldots}
 \end{array}
\]
where $\neg\beta \not\in \mh' \subseteq \mh$, $\Ga = \Ga'\cup\{\gamma\to\delta\}$. The new hybrid tree is obtained by removing from the hybrid tree described above all the sequents from, and not including, $\mh' \h \Ga', \gamma\to\delta, \neg\alpha \Rightarrow \neg\beta$ up to, and including, $\mh \h \Ga, \neg\alpha \Rightarrow \neg\beta$. We can now apply backwards one between $\nthree$ and $\nfour$ to the first of these sequents, depending on which of the side conditions are satisfied (we may need some contractions).

Suppose that below the conclusion the hybrid tree has the form 

\[
 \begin{array}{ccccccccc}
 \infer[\nr]{\mh \h \Ga, \neg\alpha \Rightarrow \neg\beta}{\Ga, \neg\alpha, \beta \Rightarrow \alpha \qquad \Ga, \neg\alpha, \alpha \Rightarrow \beta}
 \\
 \vdots
 \smallskip \\
 \infer[\nthree]{\mh' \h \Ga, \neg\alpha \Rightarrow \neg\beta}{(\neg\beta, \mh') \h \Ga, \neg\alpha \Rightarrow \alpha \qquad \ldots}
 \end{array}
\]
where $\neg\beta \not\in \mh'$ and $\mh'\subseteq\mh$. The new hybrid tree is obtained by removing from the hybrid tree described above all the sequents from, and not including, $\mh' \h \Ga, \neg\alpha \Rightarrow \neg\beta$ up to, and including, $\mh \h \Ga, \neg\alpha \Rightarrow \neg\beta$. We can now apply backwards $\nthree$ to the first of these sequents (we may need some contractions).

Suppose that below the conclusion the hybrid tree has the form 

\[
 \begin{array}{ccccccccc}
 \infer[\nr]{\mh \h \Ga, \neg\alpha \Rightarrow \neg\beta}{\Ga, \neg\alpha, \beta \Rightarrow \alpha \qquad \Ga, \neg\alpha, \alpha \Rightarrow \beta}
 \\
 \vdots
 \smallskip \\
 \infer[\nfour]{\mh' \h \Ga, \neg\alpha \Rightarrow \neg\beta}{(\neg\beta, \mh') \h \Ga, \neg\alpha \Rightarrow \alpha \qquad \ldots}
 \end{array}
\]
where $\neg\beta \not\in \mh'$ and $\mh'\subseteq\mh$. The new hybrid tree is obtained by removing from the hybrid tree described above all the sequents from, and not including, $\mh' \h \Ga, \neg\alpha \Rightarrow \neg\beta$ up to, and including, $\mh \h \Ga, \neg\alpha \Rightarrow \neg\beta$. We can now apply backwards $\nfour$ to the first of these sequents (we may need some contractions).

Suppose that below the conclusion the hybrid tree has the form 

\[
 \begin{array}{ccccccccc}
 \infer[\nr]{\mh \h \Ga, \neg\alpha \Rightarrow \neg\beta}{\Ga, \neg\alpha, \beta \Rightarrow \alpha \qquad \Ga, \neg\alpha, \alpha \Rightarrow \beta}
 \\
 \vdots
 \smallskip \\
 \infer[\nefr]{\mh' \h \Ga, \neg\alpha \Rightarrow \neg\beta}{(\neg\beta, \mh') \h \Ga, \neg\alpha \Rightarrow \alpha}
 \end{array}
\]
where $\neg\beta \not\in \mh'$ and $\mh'\subseteq\mh$. The new hybrid tree is obtained by removing from the hybrid tree described above all the sequents from, and not including, $\mh' \h \Ga, \neg\alpha \Rightarrow \neg\beta$ up to, and including, $\mh \h \Ga, \neg\alpha \Rightarrow \neg\beta$. We can now apply backwards one between $\nthree$ and $\nfour$ to the first of these sequents (we may need some contractions).

Suppose now that the last step in the derivation is an application of $\nefr$. If the side condition of the history rule $\nefr$ is satisfied, we simply add the appropriate history to the premises and possibly apply contraction. Otherwise, we have the following situation:

\[
 \begin{array}{ccccccccc}
 \vdots
 \smallskip \\
 \infer[\nefr]{\mh \h \Ga, \neg\alpha \Rightarrow \neg\beta}{\Ga, \neg\alpha \Rightarrow \alpha}
 \end{array}
\]
with $\neg\beta\in\mh$. Suppose that below the conclusion the hybrid tree has the form 

\[
 \begin{array}{ccccccccc}
 \infer[\nefr]{\mh \h \Ga, \neg\alpha \Rightarrow \neg\beta}{\Ga, \neg\alpha \Rightarrow \alpha}
 \\
 \vdots
 \smallskip \\
 \infer[\impllr]{\mh' \h \Ga', \gamma\to\delta, \neg\alpha \Rightarrow \neg\beta}{(\neg\beta, \mh') \h \Ga', \gamma\to\delta, \neg\alpha \Rightarrow \gamma \qquad \ldots}
 \end{array}
\]
where $\neg\beta \not\in \mh' \subseteq \mh$ and $\Ga = \Ga'\cup\{\gamma\to\delta\}$. The new hybrid tree is obtained by removing from the hybrid tree described above all the sequents from, and not including, $\mh' \h \Ga', \gamma\to\delta, \neg\alpha \Rightarrow \neg\beta$ up to, and including, $\mh \h \Ga, \neg\alpha \Rightarrow \neg\beta$. We can now apply backwards $\nefr$ to the first of these sequents (we may need some contractions).

Suppose that below the conclusion the hybrid tree has the form 

\[
 \begin{array}{ccccccccc}
 \infer[\nefr]{\mh \h \Ga, \neg\alpha \Rightarrow \neg\beta}{\Ga, \neg\alpha \Rightarrow \alpha}
 \\
 \vdots
 \smallskip \\
 \infer[\nthree]{\mh' \h \Ga, \neg\alpha \Rightarrow \neg\beta}{(\neg\beta, \mh') \h \Ga, \neg\alpha \Rightarrow \alpha \qquad \ldots}
 \end{array}
\]
where $\neg\beta \not\in \mh'$ and $\mh'\subseteq\mh$. The new hybrid tree is obtained by removing from the hybrid tree described above all the sequents from, and not including, $\mh' \h \Ga, \neg\alpha \Rightarrow \neg\beta$ up to, and including, $\mh \h \Ga, \neg\alpha \Rightarrow \neg\beta$. We can now apply backwards $\nefr$ to the first of these sequents (we may need some contractions).

Suppose that below the conclusion the hybrid tree has the form 

\[
 \begin{array}{ccccccccc}
  \infer[\nefr]{\mh \h \Ga, \neg\alpha \Rightarrow \neg\beta}{\Ga, \neg\alpha \Rightarrow \alpha}
 \\
 \vdots
 \smallskip \\
 \infer[\nfour]{\mh' \h \Ga, \neg\alpha \Rightarrow \neg\beta}{(\neg\beta, \mh') \h \Ga, \neg\alpha \Rightarrow \alpha \qquad \ldots}
 \end{array}
\]
where $\neg\beta \not\in \mh'$ and $\mh'\subseteq\mh$. The new hybrid tree is obtained by removing from the hybrid tree described above all the sequents from, and not including, $\mh' \h \Ga, \neg\alpha \Rightarrow \neg\beta$ up to, and including, $\mh \h \Ga, \neg\alpha \Rightarrow \neg\beta$. We can now apply backwards $\nefr$ to the first of these sequents (we may need some contractions).

Suppose that below the conclusion the hybrid tree has the form 

\[
 \begin{array}{ccccccccc}
 \infer[\nefr]{\mh \h \Ga, \neg\alpha \Rightarrow \neg\beta}{\Ga, \neg\alpha \Rightarrow \alpha}
 \\
 \vdots
 \smallskip \\
 \infer[\nefr]{\mh' \h \Ga, \neg\alpha \Rightarrow \neg\beta}{(\neg\beta, \mh') \h \Ga, \neg\alpha \Rightarrow \alpha}
 \end{array}
\]
where $\neg\beta \not\in \mh'$ and $\mh'\subseteq\mh$. The new hybrid tree is obtained by removing from the hybrid tree described above all the sequents from, and not including, $\mh' \h \Ga, \neg\alpha \Rightarrow \neg\beta$ up to, and including, $\mh \h \Ga, \neg\alpha \Rightarrow \neg\beta$. We can now apply backwards $\nefr$ to the first of these sequents (we may need some contractions).

Suppose now that the last inference rule is $\copc$. If the side condition of the history rule $\copcone$ is satisfied, we simply add the appropriate history to the premise. Else, we have 

\[
 \begin{array}{ccccccccc}
 \vdots
 \smallskip \\
 \infer[\copc]{\mh \h \Ga, \neg\alpha \Rightarrow \neg\beta}{\Ga, \neg\alpha, \beta \Rightarrow \alpha}
 \end{array}
\]

\noindent where $\beta \in \Ga\cup\{\neg\alpha\}$. If $\neg\beta \not \in \mh$, we simply apply $\copctwo$ backwards and add the appropriate history (we may need an application of contraction). If $\neg\beta \in \mh$, we need to consider two possibilities. 

Suppose that below the conclusion the hybrid tree has the form 

\[
 \begin{array}{ccccccccc}
 \infer[\copc]{\mh \h \Ga, \neg\alpha \Rightarrow \neg\beta}{\Ga, \neg\alpha, \beta \Rightarrow \alpha}
 \\
 \vdots
 \smallskip \\
 \infer[\impllr]{\mh' \h \Ga', \gamma\to\delta, \neg\alpha \Rightarrow \neg\beta}{(\neg\beta, \mh') \h \Ga', \gamma\to\delta, \neg\alpha \Rightarrow \gamma \qquad \ldots}
 \end{array}
\]

\noindent where $\neg\beta \not\in \mh' \subseteq \mh$ and $\Ga = \Ga'\cup\{\gamma\to\delta\}$. Similarly to the previous case, the new hybrid tree is obtained by removing from the hybrid tree described above all the sequents from, and not including, $\mh' \h \Ga', \gamma\to\delta, \neg\alpha \Rightarrow \neg\beta$ up to, and including, $\mh \h \Ga, \neg\alpha \Rightarrow \neg\beta$. We can now apply $\copctwo$ backwards to the first of these sequents (we may need some contractions). 

Suppose now that below the conclusion the hybrid tree has the form 

\[
 \begin{array}{ccccccccc}
 \infer[\copc]{\mh \h \Ga, \neg\alpha \Rightarrow \neg\beta}{\Ga, \neg\alpha, \beta \Rightarrow \alpha}
 \\
 \vdots
 \smallskip \\
 \infer[\copctwo]{\mh' \h \Ga, \neg\alpha \Rightarrow \neg\beta}{(\neg\beta, \mh') \h \Ga, \neg\alpha \Rightarrow \alpha}
 \end{array}
\]

\noindent where $\neg\beta \not\in \mh'$ and $\mh'\subseteq\mh$. The new hybrid tree can be obtained now by removing from the old hybrid tree all the sequents from, and including, $(\neg\beta, \mh') \h \Ga, \neg\alpha \Rightarrow \alpha$ up to, and including, $\mh \h \Ga, \neg\alpha \Rightarrow \neg\beta$. We can again proceed with an application of $\copctwo$ backwards to the first newly obtained sequent (we may need some contractions).

Finally, suppose that the last inference rule is $\anr$. If the side condition of the history rule $\anr$ is satisfied, we simply add the appropriate history to the premise. Otherwise, from the point of view of looping, both the premise and the conclusion are the same and we are in the presence of a trivial loop. The new hybrid tree is simply the old one with the premise obtained by contraction. 
\end{proof}

\begin{theorem}
Backwards proof search in the calculus $\m{G3}^{Hist}$ is terminating.
\end{theorem}

\begin{proof}
Consider the weight of a sequent defined as a triple 
\[
(k-n,k-m,w)
\] 
of natural numbers, where $k$ is the number of subformulas of $\Gamma \cup \{\varphi\}$ computed as a set, $n$ is the number of formulas of $\Gamma$ computed as a set, $m$ is the number of formulas of $\mh$, and $w$ is the weight of $\varphi$. It is easy to check that all the rules are such that the weight for the premises is lower than the one for the conclusion. 
\end{proof}

We exploit now the termination property of our sequent calculus and establish a bound on the complexity of the decision problem for the considered logical systems.

\begin{theorem}\label{t:pspace}
There exists an algorithm for backwards proof search in the calculus $\m{G3}^{Hist}$ that terminates in polynomial space.
\end{theorem}

\begin{proof}
Given a formula $\varphi$, we want to search for a proof of the sequent $\emptyset\ |\ \emptyset \Rightarrow \varphi$. We use a standard argument, inspecting the proof-search tree. If the weight of $\varphi$ is $w$, the size of $\emptyset\ |\ \emptyset \Rightarrow \varphi$ is (at most) $(w, w, w)$, and decreases with every step of the search. Hence, it is enough to observe the following: the length of a branch in the proof-search tree is bounded by $w^3$; the branching degree of the tree is bounded by $2w$; the number of formulas contained in any sequent occurring during the search can be seen to be bounded by $O(w^2)$. Therefore, the result follows. More precisely, as we in fact always need to remember just one whole branch, an upper bound of the needed space is $O(w^5)$.
\end{proof}

Statman provides, in his seminal paper~\cite{statman1979intuitionistic}, a polynomial translation of intuitionistic logic into its implicational fragment. As {\sf IPC} is PSPACE-hard, so is its implicational fragment (see also~\cite{R08}). Therefore, any conservative extension of the implicational fragment of {\sf IPC} is PSPACE-hard as well, including logics considered in this paper\footnote{It has been known before that {\sf MPC} is PSPACE-complete, and we include it here for the sake of completeness. Statman's polynomial translation from~\cite{statman1979intuitionistic} can in fact be seen as one from {\sf IPC} into \mpc. Another polynomial translation of {\sf IPC} into \mpc, which makes use of the constant $f$, was found in a letter from Johansson to Heyting from 1935~\cite{vandermolen} and is discussed in~\cite{colacito2016subminimal}. Note that {\sf MPC} can be in turn translated into {\sf CoPC} (see the remark after Proposition 4.2), but this translation is not polynomial.}. To sum up:

\begin{corollary}
The satisfiability problem for either one of the logical systems {\sf N}, {\sf NeF}, {\sf CoPC} or {\sf MPC} is PSPACE-complete. 
\end{corollary}

\begin{proof}
The inclusion in PSPACE is Theorem~\ref{t:pspace}. Hardness follows from the fact that all the logics we consider are conservative extensions of the implication fragment of intuitionistic logic, known to be PSPACE-hard~\cite{statman1979intuitionistic,R08}.
\end{proof}

\section{Conclusions and Further Research}\label{s:conclusion}

The question whether the logics addressed in this paper enjoy the property of uniform interpolation remains open. We conclude the article with some comments and remarks on this problem.

Uniform interpolation is a strengthening of the Craig interpolation property, claiming that for each formula $\phi(\vec{q},\vec{p})$ there exists a post-interpolant $\exists \vec{q}\phi(\vec{p})$ such that: (1) $\phi\to\exists \vec{q}\phi$ is provable, and (2) for each $\psi(\vec{p},\vec{r})$ with $\phi\to\psi$ provable, also $\exists \vec{q}\phi\to\psi$ is provable. Symmetrically, for each $\psi(\vec{p},\vec{r})$ there exists a pre-interpolant $\forall \vec{r}\psi(\vec{p})$ such that: (1) $\forall \vec{r}\psi\to\psi$ is provable, and (2) for each $\phi(\vec{q},\vec{p})$ with $\phi\to\psi$ provable, $\phi\to\forall\vec{r}\psi$ is provable as well. Note that here we deliberately confuse the existence of a post-interpolant and a pre-interpolant with the possibility of simulating propositional quantification.
In terms of the Craig interpolation property, uniform interpolation means that there are a least ($\exists \vec{q}\phi$) and a greatest ($\forall \vec{r}\psi$) interpolant with respect to `provability order', for each provable implication $\phi(\vec{q},\vec{p})\to\psi(\vec{p},\vec{r})$. The interpolants are unique up to provable equivalence.

Uniform interpolation for intuitionistic propositional logic has been originally proved by Pitts in~\cite{Pitts92}, using proof-theoretical means. As for positive logic, the pre-interpolation property does not hold in full generality. Nevertheless, the `failing' cases are of the following form.~Given the formula $p \to r$, its pre-interpolant $\forall r(p \to r)$ in intuitionistic logic is equivalent to the formula $\neg p$, and there is no formula in the language of positive logic containing only $p$ that does the job: for every formula $\phi(p,q)$ in the positive language, $\phi(p,q)\to (p \to r)$ is not valid. This is a typical pathological case, and in all the other cases, the pre-interpolant exists. This result was proved in~\cite{dJZ13}, where de Jongh and Zhao obtained uniform interpolation for positive logic, and for {\sf MPC} (in the language with the constant $f$); the {\sf MPC} uniform interpolant is obtained from the one of intuitionistic logic, by means of a clever construction of a positive content $\varphi^+$ associated to a formula $\varphi$.

We focus here on the possibility of proving uniform interpolation for the logic defined by the contraposition axiom. Also in this case, pre-interpolants do not always exist. In fact, the same example as in the case of positive logic does the job: if we consider again $\forall r(p \to r)$, no formula containing only the variable $p$ can work as an interpolant. Similarly to the {\sf MPC} case, it is also true that $\phi(p,q)\to (p \to r)$ does not hold for any $\phi(p,q)$ in the language. One is therefore tempted to conjecture that the revised {\sf MPC} version of uniform (pre-)interpolation could go through in this setting, and try to derive the uniform interpolation property for the logic {\sf CoPC} (and possibly, for {\sf NeF}) from that of {\sf MPC}. However, some simple differences between the two cases are clear: for instance, observe that $\forall r(p\to\neg r)\leftrightarrow \neg p$ holds in {\sf MPC}, while $\forall r(p\to\neg r) \leftrightarrow (p\to\neg p)$ holds only in the presence of the contraposition axiom, and moreover $p\to\neg p$ is the translation of the {\sf MPC} interpolant by the $\sim$ translation: $(\neg p)^{\sim} = (p\to\neg p)$ (see the remark after Proposition~\ref{p:negneg}). 
It would however be naive to hope that the uniform interpolants in the contraposition case can always be obtained by translation. In fact, a counterexample is given by the following observation: it is possible to construct a uniform pre-interpolant $\forall r(\neg p\wedge q\to\neg r)$ for the formula $(\neg p\wedge q\to\neg r)$ in the logic defined by the contraposition axiom, such interpolant being of the form $\neg p\to (q\to\neg q)$; however, the latter is not (equivalent to) the translation of the corresponding {\sf MPC} interpolant $\neg p\to \neg q$. 

We have not been able so far to prove or refute the uniform interpolation property for any of the logics {\sf N}, {\sf NeF}, and {\sf CoPC}, and we leave it for further considerations, algebraic or proof-theoretic.

 



\end{document}